\documentclass[12pt]{amsart}
\usepackage{amsmath}
\usepackage{amssymb}
\usepackage[all]{xy}
\usepackage{longtable}
\usepackage[osf,sc]{mathpazo}
\usepackage{euscript}
\usepackage{calrsfs}

\newtheorem{theorem}[equation]{Theorem}
\newtheorem{lemma}[equation]{Lemma}
\newtheorem{proposition}[equation]{Proposition}
\newtheorem{corollary}[equation]{Corollary}
\newtheorem{conjecture}[equation]{Conjecture}
\newtheorem{definition}[equation]{Definition}

\theoremstyle{remark}
\newtheorem{remark}[equation]{Remark}

\numberwithin{equation}{subsection}

\newcommand{\FF}{\mathbb{F}}
\newcommand{\ZZ}{\mathbb{Z}}
\newcommand{\QQ}{\mathbb{Q}}

\newcommand{\TT}{\mathbb{T}}
\newcommand{\GG}{\mathbb{G}}

\newcommand{\CC}{\mathbb{C}}

\newcommand{\NN}{\mathbb{N}}

\newcommand{\bm}{\mathbf{m}}

\newcommand{\bx}{\mathbf{x}}

\newcommand{\bu}{\mathbf{u}}
\newcommand{\bv}{\mathbf{v}}

\newcommand{\bC}{\mathbf{C}}

\newcommand{\bz}{\mathbf{z}}

\newcommand{\cL}{\mathcal{L}}
\newcommand{\cF}{\mathcal{F}}

\DeclareMathAlphabet{\matheur}{U}{eur}{m}{n}

\newcommand{\fQ}{\mathfrak{Q}}
\newcommand{\fs}{\mathfrak{s}}

\DeclareMathOperator{\Ker}{Ker} \DeclareMathOperator{\GL}{GL}
\DeclareMathOperator{\Mat}{Mat}

 \DeclareMathOperator{\wt}{wt}
\DeclareMathOperator{\Ext}{Ext}  
\DeclareMathOperator{\Li}{Li}

\DeclareMathOperator{\BC}{BC}
\DeclareMathOperator{\den}{den}
\DeclareMathOperator{\Eu}{Euler}

\newcommand{\ok}{\overline{k}}

\newcommand{\tr}{\mathrm{tr}}

\newcommand{\power}[2]{{#1 [\![ #2 ]\!]}}
\newcommand{\laurent}[2]{{#1 (\!( #2 )\!)}}

\begin{document}
\title[Eulerian multizeta values in positive characteristic]{{\large{ An effective criterion for Eulerian multizeta values\\ in positive characteristic}}}

\author{Chieh-Yu Chang}
\address{Department of Mathematics, National Tsing Hua University, Hsinchu City 30042, Taiwan
  R.O.C.}

\email{cychang@math.nthu.edu.tw}

\author{Matthew A. Papanikolas}
\address{Department of Mathematics, Texas A\&M University, College Station, TX 77843-3368, U.S.A.}

\email{map@math.tamu.edu}

\author{Jing Yu}
\address{Department of Mathematics, National Taiwan University and TIMS, Taipei City 106, Taiwan
  R.O.C.}

\email{yu@math.ntu.edu.tw}

\thanks{The first author was partially supported by a Golden-Jade
  fellowship of the Kenda Foundation, NCTS, and MOST Grant
  102-2115-M-007-013-MY5.  The second author was partially supported
  by NSF Grant DMS-1200577. The third author was partially supported by MOST Grant 102-2119-M-002-002.}


\date{September 28, 2015}

\begin{abstract}
Characteristic $p$ multizeta values were initially studied by Thakur, who defined them as analogues of classical multiple zeta values of Euler.  In the present paper we establish an effective criterion for Eulerian multizeta values, which characterizes when a multizeta value is a rational multiple of a power of the Carlitz period.  The resulting \lq\lq $t$-motivic\rq\rq algorithm can tell whether any given multizeta value is Eulerian or not. We also prove that if $\zeta_{A}(s_{1},\ldots,s_{r})$ is Eulerian, then $\zeta_{A}(s_{2},\ldots,s_{r})$ has to be Eulerian. This was  conjectured  by Lara Rodr\'{i}guez and Thakur for the zeta-like case from numerical data. Our methods apply equally well to values of Carlitz multiple polylogarithms at algebraic points and can also be extended to determine zeta-like multizeta values.
\end{abstract}


\maketitle

\section{Introduction}
In this paper we provide an effective criterion to determine when multizeta values in positive characteristic are Eulerian. Our study is motivated by the celebrated formula of Euler on special values of the Riemann zeta function at even positive integers: for $m\in \NN$,
\[
\zeta(2m)=\frac{-B_{2m}\left( 2\pi\sqrt{-1} \right)^{2m} }{2(2m)!},
\]
where $B_{2m}\in \QQ$ are Bernoulli numbers. In particular, we have $\zeta(2m) / \bigl( 2\pi\sqrt{-1} \bigr)^{2m}\in \QQ$ for $m\in \NN$.  For an integer $n>1$, Euler's formula implies (trivially, since $\zeta(n)$ is real) that $\zeta(n)/(2\pi \sqrt{-1})^{n}$ is rational if and only if $n$ is even.

Multiple zeta values (henceforth abbreviated MZV's), initially studied by Euler as generalizations of special zeta values, are defined by the reciprocal power sums
\[
\zeta(s_{1},\cdots,s_{r})=\sum_{n_{1}>\cdots>n_{r}\geq
  1} \frac{1}{n_{1}^{s_{1}}\cdots n_{r}^{s_{r}} },
\]
where $s_{1},\ldots,s_{r}$ are positive integers with $s_{1}\geq 2$. Here $r$ is called the depth and $w:=\sum_{i=1}^{r}s_{i}$ is called the weight of the presentation $\zeta(s_{1},\ldots,s_{r})$.  We call $\zeta(s_{1},\ldots,s_{r})$ \emph{Eulerian} if the ratio $\zeta (s_{1},\ldots,s_{r}) / \bigl( 2\pi \sqrt{-1}
\bigr)^{w}$ is rational (see \cite{Thakur04}).  It is a natural question to ask if there is a criterion for determining which MZV's of depths at least $2$ are Eulerian.

Let $A$ be the polynomial ring in the variable $\theta$ over a finite field $\FF_{q}$ with quotient field~$k$. Let $A_{+}$ be the set of monic polynomials in $A$ and consider the series, for $n\in \NN$,
\[
\zeta_{A}(n):=\sum_{a\in A_{+}} \frac{1}{a^{n}}\in
\laurent{\FF_{q}}{\frac{1}{\theta}}.
\]
These values, called Carlitz zeta values, are analogues of classical special zeta values.  We note that in this non-archimedean situation the series $\zeta_{A}(1)$ does converge in $\laurent{\FF_{q}}{\frac{1}{\theta}}$. Let $\bC$ be the Carlitz module and $\tilde{\pi}$ be a fundamental period of $\bC$.  Recall that in the function field setting $\bC$ plays the role of the multiplicative group $\GG_{m}$ and $\tilde{\pi}$ plays the role of $2\pi \sqrt{-1}$. We denote by $\exp_{\bC}(z)=\sum_{i\geq 0} z^{q^{i}}/D_{i}$ the Carlitz exponential function, and by $\Gamma_{n+1}\in A$ (for non-negative integers $n$) the Carlitz factorials (see \S\ref{sec:AT polynomials} for definitions).

In \cite{Carlitz}, Carlitz derived an analogue of Euler's formula. More precisely, we write
\[
\frac{z}{\exp_{\bC}(z)}=\sum_{n\geq 0} \frac{\BC(n)}{\Gamma_{n+1}}z^{n},
\]
where $\BC(n)\in k$ are called Bernoulli-Carlitz numbers (see \cite{Goss96}). Carlitz established the formula
\begin{equation}\label{E:CarlitzFormula}
\zeta_{A}(n)=\frac{\BC(n)}{\Gamma_{n+1}}\tilde{\pi}^{n}
\end{equation}
if $n\in \NN$ is \emph{even} (i.e., $(q-1)| n$).  We note that $\tilde{\pi}^{n}\in \laurent{\FF_{q}}{\frac{1}{\theta}}$ if and only if $n$ is \emph{even}, and so Carlitz's result implies that $\zeta_{A}(n)/\tilde{\pi}^{n}\in k$ if and only if $n$ is \emph{even}.

In \cite{AT90}, Anderson and Thakur related $\zeta_{A}(n)$ to the last coordinate of the logarithm of $\bC^{\otimes n}$ (the $n$-th tensor power of the Carlitz module viewed as a Carlitz-Tate $t$-motive) at an explicitly constructed integral point $Z_{n}$ (see \S\ref{sec:AT special points}). As a consequence, one has that the rationality of  $\zeta_{A}(n)/\tilde{\pi}^{n}$ is equivalent to $Z_{n}$ being $\FF_{q}[t]$-torsion. In this case, it is clearly described when $Z_{n}$ is $\FF_{q}[t]$-torsion, and more precisely we have that $Z_{n}$ is an $\FF_{q}[t]$-torsion point if and only if $n$ is \emph{even} (see \cite[Prop.~1.11.2, Cor.~3.8.4]{AT90} and \cite[Thm.~3.1]{Yu91}). On the other hand, Yu showed that the transcendence of $\zeta_{A}(n)/\tilde{\pi}^{n}$ is equivalent to $Z_{n}$ being non-torsion (see~\cite[Cor.~2.6]{Yu91}), whence deriving that  $\zeta_{A}(n)/\tilde{\pi}^{n}$ is algebraic over $k$ if and only if  $\zeta_{A}(n)/\tilde{\pi}^{n}$ is in $k$.

For $\fs=(s_{1},\dots, s_{r})\in \NN^{r}$, characteristic $p$ multizeta values $\zeta_{A}(\fs)$, defined by Thakur~\cite{Thakur04}, are generalizations of Carlitz zeta values.  We set
\[
  \zeta_{A}(\fs) := \sum \frac{1}{a_1^{s_1} \cdots
  a_r^{s_r}} \in \laurent{\FF_{q}}{\frac{1}{\theta}},
\]
where the sum is taken over $r$-tuples of monic polynomials $a_1, \dots, a_r$ with $\deg a_1 > \cdots > \deg a_r$, $r$ is called the depth and $w:=s_1+\cdots+s_r$ is the weight of the presentation $\zeta_{A}(\fs)$. These values are known to be non-vanishing by Thakur~\cite[Thm.~4]{T09a}. As in the classical case, Thakur called $\zeta_{A}(\fs)$ \emph{Eulerian} if the ratio $\zeta_{A}\left(\fs\right)/\tilde{\pi}^{w}$ is in $k$.
We mention that one encounters here the Eulerian multizeta values such as $\zeta_{A}(q-1, (q-1)^2)$, or $\zeta_{A}(q-1, (q-1)q, \ldots, (q-1)q^{r-1})$  (see~\cite{Thakur09b, LRT13, Ch14}), as compared with the classical Eulerian values $\zeta(2m, 2m)$, $\zeta(2, 2, \cdots, 2)$. In contrast to the classical story, we already know that these ratios $\zeta_{A}(\fs)/{\tilde{\pi}}^w$  are either rational or transcendental over $k$.  Indeed, by \cite[Cor.~2.3.3]{C14} we have that either $\zeta_{A}(\fs)/{\tilde{\pi}}^w$ is in $k$ or $\zeta_{A}(\fs)$ and $\tilde{\pi}$ are algebraically independent over $k$, generalizing the depth one results of \cite{Yu97, CY07}. However the ``irrationality'' remains a subtle question, i.e. verifying that a given specific \emph{even} weight multizeta
 value of depth $r>1$ is not Eulerian.

The main result of the present paper (Theorem~\ref{T:EffectiveCriterion}) is to give an effective criterion for Eulerian multizeta values of arbitrary depth. Inspired by Anderson-Thakur~\cite{AT09}, for any $r$-tuple $\fs=(s_{1},\ldots,s_{r})\in \NN^{r}$ we first explicitly construct an abelian $t$-module $E':=E_{\fs}'$ defined over $A$, which is a higher dimensional analogue of a Drinfeld module introduced by Anderson~\cite{A86}, and an integral point $\bv_{\fs} \in E'(A)$  such that $\zeta_{A}(\fs)$ is Eulerian if and only if $\bv_{\fs}$ is an  $\FF_{q}[t]$-torsion point in $E'(A)$. Furthermore, whenever $\zeta_A(\fs)$ is Eulerian we find an
explicit polynomial $a_{\fs}\in \FF_{q}[t]$ that annihilates the integral point $\bv_{\fs}$. This allows us to establisht an algorithm for determining when a given MZV is Eulerian or non-Eulerian.    When $r=1$, for each $s\in \NN$ the special point $\bv_{s}$ is the same as the special point $Z_{s}$ introduced by Anderson-Thakur previously.

In the classical case, Brown~\cite[Thm.~3.3]{B12b} gave a sufficient condition for Eulerian MZV's in terms of motivic multiple zeta values, which are functions defined on the motivic period torsor for the motivic Galois group of the mixed Tate motives over~$\ZZ$, and whose images under the period map are the multiple zeta values in question. Given any $\zeta(s_1,\ldots,s_{r})$ with even weight $N$, if the corresponding motivic multiple zeta value $\zeta^{m}(s_{1},\ldots,s_{r})$ is trivial under the operator $D_{<N}$ given in \cite[(3.2)]{B12b}, then Brown proves that $\zeta(s_{1},\ldots,s_{r})$ is Eulerian. We note that Brown's condition is expected to be necessary for Eulerian MZV's but it is still a conjecture in the classical transcendence theory. We further mention that there is a way in principle to check whether the action of $D_{<N}$ on $\zeta^{m}(s_{1},\ldots,s_{r})$ is vanishing by applying \cite[(3.4)]{B12b}, but  it is not completely effective. We thank Brown for correspondence regarding this effectivity issue, related details can be located in \cite{B12a}.

Even in the classical case to date there is no conjecture that describes Eulerian MZV's precisely in terms of $s_{1},\ldots,s_{r}$. Having our algorithm it seems still difficult to tell  when the integral point $\bv_{\fs}$ is an  $\FF_{q}[t]$-torsion point in $E'(A)$ directly in terms of $s_{1},\ldots,s_{r}$ alone.  However an implementation of the algorithm in this paper does reveal a description of Eulerian MZV's inductively through the tuple $(s_{1},\ldots,s_{r})$  which will be discussed in \S\ref{sec:algorithm}. In particular a notable  consequence of the main result is the fact  that if $\zeta_{A}(s_{1},\ldots,s_{r})$ is Eulerian, then the $r-1$ MZV's, $\zeta_{A}(s_{2},\ldots,s_{r}),\ldots,\zeta_{A}(s_{r})$ are simultaneously Eulerian (see Corollary~\ref{C:SimuEulerian}). For the classical MZV's, Brown's theorem on a sufficient condition for Eulerian MZV's implies that $\zeta(3,1,\ldots,3,1)$ is Eulerian (see~\cite[Rem.~4.8]{B12a}).  However, $\zeta(1)$ does not converge and so a naive analogue of the truncation result for the classical Eulerian MZV's does not make  sense. It would be interesting to ask whether some sort analogue of the characteristic $p$ truncation  is nevertheless valid for the classical Eulerian MZV's without $1$ occurring in the coordinates.

The methods of constructing $t$-modules together with specific integral points which are developed in this paper also enable us to investigate similar phenomena for zeta-like multizeta values. As defined by Thakur, $\zeta_{A}(s_{1},\ldots,s_{r})$ is called \emph{zeta-like} if the ratio
\[
\zeta_{A}(s_{1},\ldots,s_{r})/\zeta_{A}(\sum_{i=1}^{r}s_{i})
\]
is in $k$ (equivalently it is algebraic over $k$ by \cite[Thm.~2.3.2]{C14}).  A criterion for zeta-like MZV's (see Theorem~\ref{T:CriterionZeta-like}) is given in terms of $\FF_{q}[t]$-linear relations for the corresponding two integral points on our $t$-modules. Here we are also able to deduce the fact that having $\zeta_{A}(s_{1},\ldots,s_{r})$ zeta-like implies that $\zeta_{A}(s_{2},\ldots,s_{r})$ must be Eulerian (see Corollary~\ref{C:zeta-like}). This property was originally conjectured by Lara  Rodr\'{i}guez and Thakur~\cite{LRT13}. We emphasize particularly that our criterion for zeta-like MZV's leads also to an effective algorithm.  This has been worked out and implemented by Kuan and Lin in~\cite{KL15}.

In \cite{C14}, the first author defined Carlitz multiple polylogarithms (abbreviated CMPL's) that are generalizations of Carlitz polylogarithms studied in~\cite{AT90}. Unlike the classical case, where there is a simple identity between multiple zeta values and multiple polylogarithms at $(1,\dots,1)$, the function field situation is more subtle.  Anderson and Thakur~\cite{AT90} showed that each Carlitz zeta value (itself a multizeta value of depth one) is a $k$-linear combination of Carlitz polylogarithms at integral points, and it is generalized in \cite{C14} that MZV's of arbitrary depth are $k$-linear combinations of Carlitz multiple polylogarithms at integral points. Following the terminology of Eulerian multizeta values, we call a nonzero value of a CMPL at an algebraic point \emph{Eulerian} if it is a $k$-multiple of $\tilde{\pi}$ raised to the power of its weight (see \S\ref{sec:CMPLs}). In Theorem~\ref{T:MainThmCMPL}, we give a criterion to determine which CMPL's at algebraic points are Eulerian.

The main idea of this work comes from the perspective of $t$-motives. To handle the $k$-linear relations among the MZV's which interest us, we manage to lift these relations in a $t$-motivic way to $\ok(t)$-linear relations among specific power series in $t$ (where $\ok$ is a fixed algebraic closure of $k$), which can be viewed as simplified analogue of the motivic MZV's in the classical theory. The key tool we use to accomplish the process is the  linear independence criterion of~\cite[Thm.~3.1.1]{ABP04} (the ``ABP-criterion'') that has been used successfully in the last decade for dealing with transcendence/algebraic independence questions in positive characteristic. The very fact that our motivic MZV's satisfy Frobenius (Galois) difference equations (by work of Anderson-Thakur~\cite{AT09}) also enables us to prove that the common denominator of the coefficients of the lifted relations is in $\FF_{q}[t]$. This denominator gives rise to linear relations for the corresponding algebraic points under the $\FF_{q}[t]$-action, and we exploit this phenomenon as much as we can in \S\S\ref{sec:preliminary}--\ref{sec: Proof of Main Thm}.

The paper is organized as follows. In \S\ref{sec:preliminary}, we first set up the necessary preliminaries and state the criterion, Theorem~\ref{T:ThmGeneral}, which equates the $\FF_{q}(\theta)$-linear dependence of values of certain special series at $t=\theta$ to $\FF_{q}[t]$-linear dependence of elements of certain $\Ext^1$-modules.  We apply \cite[Thm.~3.1.1]{ABP04} to give a proof of Theorem~\ref{T:ThmGeneral} in \S\ref{sec: Proof of Main Thm}. We then apply Theorem~\ref{T:ThmGeneral} in \S\ref{sec:applicationsMZV} to establish the criteria for Eulerian MZV's, CMPL's at algebraic points to be Eulerian and zeta-like MZV's.  Passing to $t$-modules in \S\ref{sec:t-modules} we reformulate  these  criteria. In \S\ref{sec:algorithm}, we further prove that our criterion for Eulerian MZV's yields an algorithm for determining whether any given MZV is Eulerian or non-Eulerian. A rule specifying all
Eulerian multizeta values  is drawn from the data collected using this algorithm.


\section{Preliminaries and statement of the main result}\label{sec:preliminary}
\subsection{Notation}

We adopt the following notation.
\begin{longtable}{p{0.5truein}@{\hspace{5pt}$=$\hspace{5pt}}p{5truein}}
$\FF_q$ & the finite field with $q$ elements, for $q$ a power of a
prime number $p$. \\
$\theta$, $t$ & independent variables. \\
$A$ & $\FF_q[\theta]$, the polynomial ring in the variable $\theta$ over $\FF_q$.
\\
$A_{+}$ & set of monic polynomials in A.
\\
$k$ & $\FF_q(\theta)$, the fraction field of $A$.\\
$k_\infty$ & $\laurent{\FF_q}{1/\theta}$, the completion of $k$ with
respect to the place at infinity.\\
$\overline{k_\infty}$ & a fixed algebraic closure of $k_\infty$.\\
$\ok$ & the algebraic closure of $k$ in $\overline{k_\infty}$.\\
$\CC_\infty$ & the completion of $\overline{k_\infty}$ with respect to
the canonical extension of $\infty$.\\
$|\cdot|_{\infty}$& a fixed absolute value for the completed field $\CC_{\infty}$ so that $|\theta|_{\infty}=q$.\\
$\deg$& function assigning to $x\in k_{\infty}$ its degree in $\theta$.\\
$\power{\CC_\infty}{t}$ & ring of formal power series in $t$ over $\CC_{\infty}$.\\
$\laurent{\CC_\infty}{t}$ & field of Laurent series in $t$ over $\CC_{\infty}$.
\end{longtable}

We consider the following characteristic $p$ multizeta values defined by Thakur~\cite{Thakur04}: for any $r$-tuple of positive integers $(s_{1},\dots,s_{r})\in \NN^{r}$,
\begin{equation}\label{E:Thakur MZV}
 \zeta_{A}(s_{1},\ldots,s_{r}):=\sum\frac{1}{a_{1}^{s_{1}}\cdots a_{r}^{s_{r}}}\in k_{\infty},
\end{equation}
where the sum is over $(a_{1},\ldots,a_{r})\in A_{+}^{r}$ with $\deg a_{1} > \cdots > \deg a_{r}$.  Thakur~\cite{T09a} showed that each multizeta value is non-vanishing.

\subsection{Frobenius modules}
We consider the following automorphism of the field of Laurent series over $\CC_{\infty}$, which is referred to as \emph{Frobenius twisting}:
\[
     \begin{array}{rcl}
      \laurent{\CC_\infty}{t}  & \rightarrow & \laurent{\CC_\infty}{t},\\
       f:=\sum_{i}a_{i}t^{i} & \mapsto & f^{(-1)}:=\sum_{i}{a_{i}}^{\frac{1}{q}}t^{i}. \\
     \end{array}
 \]
We extend Frobenius twisting to matrices with entries in $\laurent{\CC_\infty}{t}$ by twisting entry-wise.

We let $\bar{k}[t,\sigma]=\bar{k}[t][\sigma]$ be the non-commutative $\bar{k}[t]$-algebra generated by the new variable $\sigma$ subject to the relation
\[
\sigma f=f^{(-1)}\sigma, \quad f\in \bar{k}[t].
\]
We call a left $\ok[t,\sigma]$-module a \emph{Frobenius module} if it is free of finite rank over $\ok[t]$. Morphisms of Frobenius modules are left $\ok[t,\sigma]$-module homomorphisms. We denote by $\cF$ the category of Frobenius modules.

The trivial object of $\cF$ is denoted by $\mathbf{1}$, where the underlying space of ${\mathbf{1}}$ is $\ok[t]$ equipped with the $\sigma$-action given by
\[
\sigma (f):=f^{(-1)}, \quad f\in {\mathbf{1}}.
\]
Another example of an object in $\cF$ is the $n$-th tensor power of the Carlitz motive denoted by $\mathrm{C}^{\otimes n}$, where $n$ is a positive integer. The underlying space of $\mathrm{C}^{\otimes n}$ is $\bar{k}[t]$, on which the action of $\sigma$ is given by
\[
\sigma (f):=(t-\theta)^{n}f^{(-1)}, \quad f\in \mathrm{C}^{\otimes n}.
\]

In what follows, an object $M$ in $\cF$ is said to be defined by a matrix $\Phi\in \Mat_{r}(\ok[t])$ if $M$ is free of rank $r$ over $\ok[t]$ and the $\sigma$-action on a given $\ok[t]$-basis of $M$ is represented by the matrix $\Phi$.

For a Frobenius module $M$, we consider the tensor product $\ok(t)\otimes_{\ok[t]}M$ on which $\sigma$ acts diagonally. It follows that $\ok(t)\otimes_{\ok[t]}M$ becomes a left $\ok(t)[\sigma]$-module, where $\ok(t)[\sigma]$ is the twisted polynomial ring in $\sigma$ over $\ok(t)$ subject to the relation $\sigma h=h^{(-1)} \sigma$ for $h\in \ok(t)$. The following proposition is a slight generalization of \cite[Prop.~3.4.5]{P08}, but it is crucial while proving Theorem~\ref{T:ThmGeneral}.

\begin{proposition}\label{P:CommonDen}
For $i=1,2$, let $M_{i}$ be a Frobenius module of rank $r_i$ over $\ok[t]$ defined by a given matrix $\Phi_{i}\in \Mat_{r_{i}}(\ok[t])$ with respect to a fixed $\ok[t]$-basis  $\bm_{i}$ of $M_{i}$. Put $\mathcal{M}_{i}:=\ok(t)\otimes_{\ok[t]}M_{i}$ for $i=1,2$, and let $f:\mathcal{M}_{1}\rightarrow \mathcal{M}_{2}$ be a homomorphism of left $\ok(t)[\sigma]$-modules. With respect to the bases $1\otimes \bm_{1}$ and $1\otimes \bm_{2}$, $f$ is represented by a matrix $F\in \Mat_{r_{1}\times r_{2}}(\ok[t])$. Suppose that $\det \Phi_{i}=c_{i}(t-\theta)^{s_{i}}$ for some $c_{i}\in \ok^{\times}$ and $s_{i}\in \ZZ_{\geq 0}$ for $i=1,2$. Then the common denominator of the entries of $F$ is in $\FF_{q}[t]$.
 \end{proposition}
\begin{proof} (cf. proof of \cite[Prop.~3.4.5]{P08}) Note that since $f$ is $\ok(t)[\sigma]$-linear, we have that
\[F^{(-1)}\Phi_{2} =\Phi_{1}F. \] For a matrix $B\in \Mat_{r\times s}(\ok(t))$, we denote by $\den(B)$ the monic least common multiple of the denominators of the entries of $B$. Since by hypothesis $\det \Phi_{2}=c_2 (t-\theta)^{s_{2}}$ for some $c_{2}\in \ok^{\times}$ and $s_{2}\geq 0$, we find that
\[ \den (F) (t-\theta)^{s_{2}}F^{(-1)}=\den (F) (t-\theta)^{s_{2}}\Phi_{1}F\Phi_{2}^{-1}\in \Mat_{r_{1}\times r_{2}}(\ok[t])   .\]
It follows that $\den(F^{(-1)})$ divides $\den (F) (t-\theta)^{s_{2}}$. As we have $\den(F^{(-1)})=\den(F)^{(-1)}$, it follows that $\deg_{t} \left(\den(F^{(-1)})\right)=\deg_{t}\left(\den(F)^{(-1)}\right) $. Therefore, it suffices to show that $\den(F^{(-1)})$ is relatively prime to $t-\theta$, since then $ \den(F^{(-1)})=\den(F),$ which implies $\den(F)\in \FF_{q}[t]$.

If $t-\theta$ divides  $\den(F^{(-1)})$, then $t-\theta^{q}$ divides $\den(F)$.  This forces $t-\theta^{q}$ to divide $\den\left(\Phi_{1}F\right)$, since otherwise $t-\theta^{q}$ would divide $\det\Phi_{1}=c_{1}(t-\theta)^{s_{1}}$.  Likewise, $t-\theta^{q}$ divides $\den\left(\Phi_{1}F\Phi_{2}^{-1} \right)=\den(F^{(-1)})$. Repeating the same argument above shows that $\den(F^{(-1)})$ is divisible by each of
\[t-\theta,t-\theta^{q},t-\theta^{q^{2}},\ldots  ,\]
whence we obtain a contradiction since $\den(F^{(-1)})\in\ok[t]$.
\end{proof}

\subsection{Frobenius modules connected to Carlitz multiple polylogarithms}

Given a polynomial $Q:=\sum_{i}a_{i}t^{i}\in \bar{k}[t]$, its Gauss norm is defined as  $\lVert Q\rVert_{\infty}:={\rm{max}}_{i}\left\{|a_{i}|_{\infty} \right\}$.  For a $r$-tuple $\fs = (s_{1},\ldots,s_{r})\in \NN^{r}$, we let $\fQ := (Q_{1},\ldots,Q_{r})\in \bar{k}[t]^{r}$ satisfy the hypothesis that as $0\leq i_{r}< \cdots< i_{1}\rightarrow \infty$,
\begin{equation}\label{E:HopythesisQ}
\left( \lVert Q_{1}\rVert_{\infty} \bigm/ \lvert \theta \rvert^{qs_{1}/(q-1)}_{\infty}\right)^{q^{i_{1}}}\cdots
\left( \lVert Q_{r}\rVert_{\infty} \bigm/ \lvert \theta \rvert^{qs_{r}/(q-1)}_{\infty}\right)^{q^{i_{r}}}
\rightarrow 0.
\end{equation}

Throughout this paper, we fix a fundamental period $\tilde{\pi}$ of the Carlitz module $\bC$ (see \cite{Goss96, Thakur04}). We put
\[
\Omega(t):=(-\theta)^{\frac{-q}{q-1}} \prod_{i=1}^{\infty} \biggl(
1-\frac{t}{\theta^{q^{i}}} \biggr)\in \power{\CC_{\infty}}{t},\]
where $(-\theta)^{\frac{1}{q-1}}$ is a suitable choice of  $(q-1)$-st root of $-\theta$ so that $\frac{1}{\Omega(\theta)}=\tilde{\pi}$ (cf.~\cite{ABP04, AT09}). We note that $\Omega$ satisfies the functional equation $\Omega^{(-1)}=(t-\theta)\Omega$.  Given $r$-tuples $\fs$ and $\fQ$ as above, we define the series
\begin{equation}\label{E:LsQ}
\mbox{\small
    $\begin{aligned}
     \cL_{\mathfrak{s},\mathfrak{Q}}(t) &:= \sum_{i_{1}>\cdots>i_{r}\geq 0} \bigl(\Omega^{s_{r}}Q_{r} \bigr)^{(i_{r})}\cdots \bigl(\Omega^{s_{1}}Q_{1} \bigr)^{(i_{1})} \\
     &= \Omega^{s_{1} + \cdots + s_{r}}\sum_{i_{1}>\cdots>i_{r}\geq 0} \frac{{Q_{r}}^{(i_{r})}(t)\cdots Q_{1}^{(i_{1})}(t) }{ \bigl( (t-\theta^{q})\ldots(t-\theta^{q^{i_{r}}})  \bigr)^{s_{r}}\ldots \bigl((t-\theta^{q})\ldots(t-\theta^{q^{i_{1}}})  \bigr)^{s_{1}}}.
    \end{aligned}$}
\end{equation}

We define $\mathcal{E}$ to be the ring consisting of formal power series $\sum_{n=0}^{\infty}a_{n}t^{n}\in \power{\ok}{t}$ such that
\[
\lim_{n\rightarrow \infty}\sqrt[n]{|a_{n}|_{\infty}}=0,\hbox{ }[k_{\infty}\left(a_{0},a_{1},a_{2},\ldots  \right):k_{\infty}]<\infty.
\]
Then any $f$ in $\mathcal{E}$ has an infinite radius of convergence with respect to $\lvert \cdot \rvert_{\infty}$, and functions in $\mathcal{E}$ are called entire functions. It is shown in \cite[Lem.~5.3.1]{C14} that the series $\cL_{\fs,\fQ}$ defined above is an entire function. We note that when $\fQ\in (\bar{k}^{\times})^{r}$ satisfies \eqref{E:HopythesisQ} then $\tilde{\pi}^{s_{1}+\cdots+s_{r}}\cL_{\fs,\fQ}(\theta)$ is the Carlitz multiple polylogarithm $\Li_{\fs}$ evaluated at the algebraic point $\mathfrak{Q}$. See \S\ref{sec:CMPLs} for additional details.

\begin{proposition}\label{P:formulaLsQ}
Let $\fs\in \NN^{r}$ and $\fQ\in \ok[t]^{r}$ satisfy the hypothesis \eqref{E:HopythesisQ}.  Then for any non-negative integer $n$, we have that
\[
\cL_{\fs,\fQ}\bigl(\theta^{q^{n}}\bigr)=\cL_{\fs,\fQ}(\theta)^{q^{n}}.
\]
\end{proposition}

\begin{proof}
The proof is essentially the same as the proof of \cite[Lem.~5.3.5]{C14} by changing $u_{i}$ to $Q_{i}$. We omit the details.
\end{proof}

Let $r$ be a positive integer.  We fix two $r$-tuples $\fs\in \NN^{r}$ and $\fQ\in \ok[t]^{r}$ satisfying \eqref{E:HopythesisQ}. We define the matrix $\Phi=\Phi_{\fs,\fQ} \in \Mat_{r+1}(\ok[t])$,
\begin{equation}\label{E:Phi s}
\Phi :=
               \begin{pmatrix}
                (t-\theta)^{s_{1}+\cdots+s_{r}}  & 0 & 0 &\cdots  & 0 \\
                Q_{1}^{(-1)}(t-\theta)^{s_{1}+\cdots+s_{r}}  & (t-\theta)^{s_{2}+\cdots+s_{r}} & 0 & \cdots & 0 \\
                 0 &Q_{2}^{(-1)} (t-\theta)^{s_{2}+\cdots+s_{r}} &  \ddots&  &\vdots  \\
                 \vdots &  & \ddots & (t-\theta)^{s_{r}} & 0 \\
                 0 & \cdots & 0 & Q_{r}^{(-1)}(t-\theta)^{s_{r}} & 1 \\
               \end{pmatrix}.
\end{equation}
Define $\Phi' = \Phi'_{\fs,\fQ}$ to be the square matrix of size $r$ cut from the upper left square of $\Phi$:
\begin{equation}\label{E:Phi s'}
\Phi' :=
\begin{pmatrix}
                (t-\theta)^{s_{1}+\cdots+s_{r}}  &  &  &  \\
                Q_{1}^{(-1)}(t-\theta)^{s_{1}+\cdots+s_{r}}  & (t-\theta)^{s_{2}+\cdots+s_{r}}   &  &  \\
                  & \ddots & \ddots &  \\
                  &  & Q_{r-1}^{(-1)}(t-\theta)^{s_{r-1}+s_{r}} & (t-\theta)^{s_{r}}  \\
\end{pmatrix}.
\end{equation}
In what follows, to avoid heavy notation we omit the subscripts $\fs$, $\fQ$ when it is clear from the context.

For $1\leq \ell <j\leq r+1  $, we define the series
\begin{equation}\label{E:Ljl}
 \cL_{j\ell}(t):=\sum_{i_{\ell}>\cdots>i_{j-1}\geq 0} \left( \Omega^{s_{j-1}} Q_{j-1}\right)^{(i_{j-1})} \cdots\left( \Omega^{s_{\ell}}Q_{\ell}\right) ^{(i_{\ell})}\in \mathcal{E},
\end{equation}
which is the same series in \eqref{E:LsQ} associated to the two tuples $(s_{\ell},\ldots,s_{j-1})$ and $(Q_{\ell},\ldots,Q_{j-1})$. Define $\Psi \in \Mat_r(\mathcal{E}) \cap \GL_r(\TT)$ by
\begin{equation}\label{E:Psi s}
\Psi:=\begin{pmatrix}
\Omega^{s_{1}+\cdots+s_{r}}& & & & & \\
\Omega^{s_{2}+\cdots+s_{r}}\cL_{21}&\Omega^{s_{2}+\cdots+s_{r}} & & & & \\
\vdots& \Omega^{s_{3}+\cdots+s_{r}} \cL_{32} &\ddots & & & \\
\vdots&\vdots &\ddots &\ddots & & \\
\Omega^{s_{r}}\cL_{r1}&\Omega^{s_{r}}\cL_{r2} & &\ddots &\Omega^{s_{r}} & \\
\cL_{(r+1),1}&\cL_{(r+1),2} &\cdots &\cdots &\cL_{(r+1),r} &1 \\
\end{pmatrix},
\end{equation}
and note that we have $\Psi^{(-1)}=\Phi\Psi$ (cf. \cite[\S 2.5]{AT09}). Let $\Psi'$ be the square matrix of size $r$ cut from the upper left square of $\Psi$. So then $\Psi'^{(-1)}=\Phi'\Psi'$. Note that $\Phi$ defines an object in $\cF$ which is a $t$-motive in the sense of \cite{P08}.

\subsection{The $\Ext^{1}$-module}

We continue the notation from the previous paragraphs.  We denote by $M$ and $M'$ the objects in $\cF$ defined by the matrices $\Phi$ and $\Phi'$ respectively. Note that $M$ fits into the short exact sequence of Frobenius modules,
\[
0\rightarrow M' \rightarrow M \twoheadrightarrow {\mathbf{1}}\rightarrow 0,
\]
and so $M$ represents a class in $\Ext_{\cF}^{1}({\mathbf{1}},M')$.  The group $\Ext_{\cF}^1(\mathbf{1},M')$ has a natural $\FF_{q}[t]$-module structure coming from Baer sum and the pushout of morphisms of $M'$. More precisely, if $M_{1}$ and $M_{2}$ represent classes in $\Ext_{\cF}^{1}({\mathbf{1}},M')$ and are defined by the two matrices respectively
\[
\Phi_{1}:=
               \begin{pmatrix}
                 \Phi' & 0 \\
                 \bv_{1} & 1 \\
               \end{pmatrix}, \quad
\Phi_{2}:=
               \begin{pmatrix}
                 \Phi' & 0 \\
                 \bv_{2} & 1 \\
               \end{pmatrix},
\]
then the Baer sum $M_{1}+_{B}M_{2}$ is the object in $\cF$ defined by the matrix
\[
\begin{pmatrix}
     \Phi' & 0 \\
           \bv_{1}+\bv_{2} & 1
    \end{pmatrix}.
\]
Furthermore, for any $a\in \FF_{q}[t]$ multiplication by $a$ induces an endomorphism of $M'$, and so the pushout $a*M_{1}\in \cF$, which is defined by the matrix
\[
\begin{pmatrix}
     \Phi' & 0 \\
     a\bv_{1} & 1
     \end{pmatrix},
\]
thus inducing a left $\FF_{q}[t]$-module structure on $\Ext_{\cF}^{1}({\mathbf{1}},M')$.

\subsection{The main theorem}
We continue with the notation as above, but assume that $r\geq 2$. We let $w:=\sum_{i=1}^{r}s_{i}$ and let $Q\in \ok[t]$ satisfy $\|Q\|_{\infty}<|\theta|_{\infty}^{{wq}/(q-1)}$. We further assume that the series $\cL_{w,Q}(t)\in \mathcal{E}$ associated to $w$ and $Q$ is non-vanishing at $t=\theta$. We let $N\in \cF$ be the Frobenius module that is defined by the matrix
\begin{equation}\label{E:PhiN}
 \begin{pmatrix}
\Phi'& {\bf{0}}\\
\bu_{w} &1\\
\end{pmatrix}\in \Mat_{r+1}(\ok[t]),
\end{equation}
 where $\bu_{w}:=\left( Q^{(-1)}(t-\theta)^{w},0,\ldots,0\right)\in \Mat_{1\times r}(\ok[t])$. Then $N$ represents a class in $\Ext_{\cF}^{1}({\bf{1}},M')$.

The following result gives a criterion for the $k$-linear dependence of the specific values $\left\{ \cL_{\fs,\fQ}(\theta),\cL_{w,Q}(\theta),1 \right\} $, which is applied to the settings of Eulerian MZV's, Eulerian CMPL's at algebraic points, and zeta-like MZV's in \S\ref{sec:applicationsMZV}.  Its proof occupies the next section.

\begin{theorem}\label{T:ThmGeneral}
Let $r\geq 2$ be a positive integer.  We fix two $r$-tuples $\fs\in \NN^{r}$ and $\fQ\in \ok[t]^{r}$ satisfying \eqref{E:HopythesisQ}. Let $M$ and $M'$ be the objects in $\cF$ defined by the matrices $\Phi$ and $\Phi'$, as in \eqref{E:Phi s} and \eqref{E:Phi s'}.  For $1\leq \ell<j\leq r+1$, we let $\cL_{j\ell}$ be defined as in \eqref{E:Ljl} and suppose that it satisfies the non-vanishing hypothesis
\begin{equation}\label{E:nonzeroHyp}
\cL_{j \ell}(\theta)\neq 0.
\end{equation}
We let $w:=\sum_{i=1}^{r}s_{i}$ and let $Q\in \ok[t]$ satisfy $\|Q\|_{\infty}<|\theta|_{\infty}^{{wq}/(q-1)}$ and $\cL_{w,Q}(\theta)\neq 0$. Let $N\in\cF$ be defined by the matrix given in (\ref{E:PhiN}). Then the following hold.
\begin{itemize}
\item[(a)]
 The set  $\left\{ \cL_{\fs,\fQ}(\theta),\cL_{w,Q}(\theta),1 \right\} $ is linearly dependent over $k$ if and only if the classes of $M$ and $N$ are $\FF_{q}[t]$-linearly dependent in $\Ext_{\cF}^{1}({\bf{1}},M')$, i.e., there exists $a,b\in \FF_{q}[t]$ (not both zero) so that $a*M+_{B}b*N$ represents a trivial class in $\Ext_{\cF}^{1}({\bf{1}},M')$.
\item[(b)] If $\left\{ \cL_{\fs,\fQ}(\theta),\cL_{w,Q}(\theta),1 \right\} $ are linearly dependent over $k$, then each of $\cL_{r+1,2}(\theta),\ldots,\cL_{r+1,r}(\theta)$ is also in $k$.
\end{itemize}
\end{theorem}

\begin{remark}\label{Rem:Coeff:a}
Let notation and hypotheses be given as above. If $c_{1}\cL_{\fs,\fQ}(\theta)+c_{2}\cL_{\fs,Q}(\theta)+c_{3}=0$ with $c_{1},c_{2},c_{3}\in k$ and $c_{1}\neq 0$, then we can find  $a, b\in \FF_{q}[t]$ with $a\neq 0$ so that  $a*M+_{B}b*N$ represents a trivial class in $\Ext_{\cF}^{1}({\bf{1}},M')$ as  can be seen from the construction of $a$ in the proof of Theorem~\ref{T:ThmGeneral}~(a)~($\Rightarrow$) (note that in this situation $f_{r+1}(\theta)\neq 0$ in  \S\S~\ref{sub:Proof of Main Thm(a)}).
\end{remark}

\begin{remark}\label{Rem:k-dep}
Note that the $k$-linear dependence of  $\left\{ \cL_{\fs,\fQ}(\theta),\cL_{w,Q}(\theta),1 \right\}$ is equivalent to the $k$-linear dependence of $\left\{ \tilde{\pi}^{w}\cL_{\fs,\fQ}(\theta) ,\tilde{\pi}^{w}\cL_{w,Q}(\theta),\tilde{\pi}^{w} \right\}$. We mention that the values $\tilde{\pi}^{w}\cL_{\fs,\fQ}(\theta),\tilde{\pi}^{w}\cL_{w,Q}(\theta)$ satisfy the MZ property with weight $w$ in the sense of \cite[Def.~3.4.1]{C14}), and hence by \cite[Prop.~4.3.1]{C14} we have that $\left\{ \cL_{\fs,\fQ}(\theta),\cL_{w,Q}(\theta),1 \right\}$ are linearly dependent over $k$ if and only if they are linearly dependent over $\ok$, and that $\cL_{\fs,\fQ}(\theta)\in k$ if and only if $\cL_{\fs,\fQ}(\theta)\in \ok$.
\end{remark}

For the applications to Eulerian MZV's and Eulerian CMPL's at algebraic points we single out the following result, which is a special case of the theorem above.

\begin{corollary}\label{Cor:CorThmGeneral}
Let notation and assumptions be given as in Theorem~\ref{T:ThmGeneral}. Then we have that $\cL_{\fs,\fQ}(\theta)(=\cL_{r+1,1}(\theta))$ is in $k$ if and only if $M$ represents a torsion element in the $\FF_{q}[t]$-module $\Ext_{\cF}^{1}({\bf{1}},M')$.
\end{corollary}

\begin{proof}
The proof of ($\Rightarrow$) is given in the case (II) of the proof of ($\Rightarrow$) of Theorem~\ref{T:ThmGeneral}(a). The proof of ($\Leftarrow$) follows from the proof of ($\Leftarrow$) of Theorem~\ref{T:ThmGeneral}(a) by putting $b=0$.
\end{proof}

\section{Proof of Theorem~\ref{T:ThmGeneral}}\label{sec: Proof of Main Thm}

\subsection{A remark}
For $a$, $b\in \FF_{q}[t]$, the Frobenius module $a*M+_{B}b*N$ is defined by the matrix
\begin{equation}\label{E:matrix-X}
  X:=\begin{pmatrix}
\Phi'& {\bf{0}}\\
\bu &1\\
\end{pmatrix}\in \Mat_{r+1}(\ok[t]),
 \end{equation} where $\bu:=\left(bQ^{(-1)}(t-\theta)^{w},0, \ldots,0,a Q_{r}^{(-1)}(t-\theta)^{s_{r}}   \right)\in \Mat_{1\times r}(\ok[t])$.  It follows that the Frobenius module $a*M+_{B}b*N$  represents a trivial class in $\Ext_{\cF}^{1}({\bf{1}},M')$ if and only if there exists $\delta_{1},\ldots,\delta_{r}\in \ok[t]$ so that
 \begin{equation}\label{E:deltaX}
\begin{pmatrix}
1& & & \\
& 1& & \\
& & \ddots& \\
\delta_{1} &\cdots &\delta_{r} & 1 \\
\end{pmatrix}^{(-1)}X=\begin{pmatrix}
\Phi'&\\
&1\\
\end{pmatrix} \begin{pmatrix}
1& & & \\
& 1& & \\
& & \ddots& \\
\delta_{1} &\cdots &\delta_{r} & 1 \\
\end{pmatrix},
\end{equation} which is equivalent to that
\begin{equation}\label{E:Eqdelta}
\begin{split}
  \delta_{1} &:= {\textstyle \delta_{1}^{(-1)}(t-\theta)^{w}+\delta_{2}^{(-1)}Q_{1}^{(-1)}(t-\theta)^{w}+bQ^{(-1)}(t-\theta)^{w}   ;} \\
   \delta_{2} &:= {\textstyle \delta_{2}^{(-1)}(t-\theta)^{s_{2}+\cdots+s_{r}}+\delta_{3}^{(-1)}Q_{2}^{(-1)}(t-\theta)^{s_{2}+\cdots+s_{r}}  ;} \\
    &\quad\vdots \\
\delta_{r-1} &:= {\textstyle \delta_{r-1}^{(-1)}(t-\theta)^{s_{r-1}+s_{r}}+\delta_{r}^{(-1)}Q_{r-1}^{(-1)}(t-\theta)^{s_{r-1}+s_{r}}  ;}\\
\delta_{r}&:={\textstyle \delta_{r}^{(-1)}(t-\theta)^{s_{r}}+aQ_{r}^{(-1)}(t-\theta)^{s_{r}}   .}
\end{split}
\end{equation}

\subsection{Proof of Theorem~\ref{T:ThmGeneral} (a)($\Rightarrow$)}\label{sub:Proof of Main Thm(a)}
Suppose that $\left\{ \cL_{\fs,\fQ}(\theta),\cL_{w,Q}(\theta),1 \right\}$ are linearly dependent over $k$. Our goal is to find $a,b\in \FF_{q}[t]$ (not both zero) and $\delta_{1},\ldots,\delta_{r}\in \ok[t]$ satisfying the equations~(\ref{E:Eqdelta}).

Define the matrix
\[
\widetilde{\Phi}:=\begin{pmatrix}
1& &\\
&\Phi & \\
&Q^{(-1)}(t-\theta)^{w},0,\ldots,0  & 1
\end{pmatrix}\in \Mat_{r+3}(\ok[t])
\]
and put
\[
\widetilde{\psi}:=\begin{pmatrix}
1\\
\Omega^{s_{1}+\cdots+s_{r}}\\
\Omega^{s_{2}+\cdots+s_{r}}\cL_{21}  \\
\vdots\\
 \cL_{r+1,1}  \\
\cL_{w,Q}
\end{pmatrix}.
\]
Then we have the difference equations $\widetilde{\psi}^{(-1)}=\widetilde{\Phi}\widetilde{\psi}$. Note that $\cL_{r+1,1}=\cL_{\fs,\fQ}$.

{\textbf{Case (I)}}. $\cL_{r+1,1}(\theta)\notin k$ (which is equivalent to $\cL_{r+1,1}(\theta)\notin \ok$ by Remark~\ref{Rem:k-dep}). Since $\left\{ \cL_{\fs,\fQ}(\theta),\cL_{w,Q}(\theta),1 \right\}$ are linearly dependent over $k$, by \cite[Thm.~3.1.1]{ABP04} there exists
\[
{\mathbf{f}}=(f_{0},f_{1},\ldots,f_{r+2})\in \Mat_{1\times (r+3)}(\ok[t])
\]
so that ${\mathbf{f}}\widetilde{\psi}=0$ and ${\mathbf{f}}(\theta)\widetilde{\psi}(\theta)=0$, which describes a nontrivial $k$-linear relation among $\left\{ \cL_{\fs,\fQ}(\theta),\cL_{w,Q}(\theta),1 \right\}$. Note that
\[
f_{1}(\theta)=\cdots=f_{r}(\theta)=0.
\]

Now we assume that $f_{r+2}(\theta)\neq 0$. Note if $f_{r+2}(\theta)=0$, then we have that
\[
\left\{\cL_{\fs,\fQ}(\theta)= \cL_{r+1,1}(\theta), 1\right\}
\]
are linearly dependent over $k$, i.e., $\cL_{r+1,1}(\theta)\in k$, and this case will be included in the case (II) below.

If we put $\tilde{\mathbf{f}}:=\frac{1}{f_{r+2}}{\mathbf{f}}\in \Mat_{1\times (r+3)}(\ok(t))$, then all entries of $\tilde{\mathbf{f}}$ are regular at $t=\theta$. Considering the Frobenius twisting-action $(\cdot)^{(-1)}$ on the equation  $\tilde{{\bf{f}}} \widetilde{\psi}=0$ and subtracting the resulting equation from  $\tilde{{\mathbf{f}}} \widetilde{\psi}=0$, we obtain
\begin{equation}\label{E:f-ftwistPhi}
\left( \tilde{{\mathbf{f}}}-\tilde{{\mathbf{f}}}^{(-1)} \widetilde{\Phi} \right) \widetilde{\psi}=0.
\end{equation}
Explicit calculations show that $\tilde{{\mathbf{f}}}-\tilde{{\mathbf{f}}}^{(-1)} \widetilde{\Phi} = ( B,B_{1},\dots,B_{r+1},0)$, where
\begin{equation}\label{E:BIdentity}
\begin{split}
B&:={\textstyle \frac{f_{0}}{f_{r+2}}-(\frac{f_{0}}{f_{r+2}})^{(-1)}}\\
  B_{1} &:= {\textstyle \frac{f_{1}}{f_{r+2}}-\left( \frac{f_{1}}{f_{r+2}} \right)^{(-1)}(t-\theta)^{w}-\left( \frac{f_{2}}{f_{r+2}} \right)^{(-1)}Q_{1}^{(-1)} (t-\theta)^{w}-Q^{(-1)}(t-\theta)^{w};} \\
   B_{2} &:= {\textstyle \frac{f_{2}}{f_{r+2}}-\left( \frac{f_{2}}{f_{r+2}} \right)^{(-1)}(t-\theta)^{s_{2}+\cdots+s_{r}}-\left( \frac{f_{3}}{f_{r+2}} \right)^{(-1)}Q_{2}^{(-1)}(t-\theta)^{s_{2}+\cdots+s_{r}};} \\
    &\quad\vdots \\
B_{r} &:= {\textstyle \frac{f_{r}}{f_{r+2}}-\left( \frac{f_{r}}{f_{r+2}} \right)^{(-1)}(t-\theta)^{s_{r}}-\left( \frac{f_{r+1}}{f_{r+2}} \right)^{(-1)}Q_{r}^{(-1)}(t-\theta)^{s_{r}}}\\
B_{r+1}&:={\textstyle \left( \frac{f_{r+1}}{f_{r+2}} \right)- \left( \frac{f_{r+1}}{f_{r+2}} \right)^{(-1)}.}
\end{split}
\end{equation}

We claim that $B=B_{1}=\cdots=B_{r+1}=0$. Assuming this claim first, we see that
\begin{equation}\label{E:fiPhis}
\mbox{\small $\begin{pmatrix}
1& & & \\
& 1& & \\
& & \ddots& \\
f_{0}/f_{r+2} &\cdots &f_{r+1}/f_{r+2} & 1 \\
\end{pmatrix}^{(-1)}\widetilde{\Phi}=\begin{pmatrix}
1& & \\
&\Phi&\\
& &1\\
\end{pmatrix} \begin{pmatrix}
1& & & \\
& 1& & \\
& & \ddots& \\
f_{0}/f_{r+2} &\cdots &f_{r+1}/f_{r+2} & 1 \\
\end{pmatrix}.$}
\end{equation}
Let $\widetilde{M}$ be the Frobenius module defined by the matrix $\widetilde{\Phi}$. Then the equation (\ref{E:fiPhis}) gives a left $\ok(t)[\sigma]$-module homomorphism between $\ok(t)\otimes_{\ok[t]}\left({\bf{1}}\oplus M\oplus {\bf{1}} \right)$ and $\ok(t)\otimes_{\ok[t]}\widetilde{M}$. It follows from Proposition~\ref{P:CommonDen} that the denominator of each $f_{i}/f_{r+2}$ is in $\FF_{q}[t]$ for $i=0,\dots,r+1$. Now we let $b\in \FF_{q}[t]$ be the common denominator of $f_{0}/f_{r+2},\ldots,f_{r+1}/f_{r+2}$, and take $\delta_{i}:=b f_{i}/f_{r+2}\in \ok[t]$ for $i=1,\ldots,r$. Note that the vanishing of  $B_{r+1}$ implies $f_{r+1}/f_{r+2}\in \FF_{q}(t)$ and hence $a:=b f_{r+1}/f_{r+2}\in \FF_{q}[t]$. Multiplying by $b$ on the both sides of (\ref{E:BIdentity}) one finds exactly the identities (\ref{E:Eqdelta}), which imply that $a*M+_{B}b*N$ represents a trivial class in $\Ext_{\cF}^{1}({\bf{1}},M')$.

To prove the claim above, we consider \eqref{E:f-ftwistPhi}, which is expanded as
\begin{equation}\label{E:equaBL}
B+B_{1}\Omega^{s_{1}+\cdots+s_{r}}+B_{2}\Omega^{s_{2}+\cdots+s_{r}}\cL_{21}+\cdots+B_{r}\Omega^{s_{r}}\cL_{r1}+B_{r+1}\cL_{r+1,1}=0.
\end{equation}
For each $1\leq \ell < j \leq r+1$ and any non-negative integer $n$, by Proposition~\ref{P:formulaLsQ}
\begin{equation}\label{E:ForLji theta}
 \cL_{j\ell}(\theta^{q^{n}})=\cL_{j\ell}(\theta)^{q^{n}},
\end{equation}
which is nonzero by hypothesis.  Since $B$ and each $B_{i}$ are rational functions in $\ok(t)$, $B$ and $B_{i}$ are defined at $t=\theta^{q^{n}}$ for sufficiently large integers $n$. We further note that $\Omega$ has a simple zero at $t=\theta^{q^{i}}$ for each positive integer $i$. Specializing (\ref{E:equaBL}) at $t=\theta^{q^{n}}$ shows that $B(\theta^{q^{n}})=B_{r+1}(\theta^{q^{n}})=0$ for any $n\gg 0$ because of $\cL_{r+1,1}(\theta)\notin  \ok$.  It follows that $B=B_{r+1}=0$.

Next, dividing \eqref{E:equaBL} by $\Omega^{s_{r}}$ and then specializing at $t=\theta^{q^{n}}$, we see from \eqref{E:ForLji theta} that $B_{r}(\theta^{q^{n}})=0$ for all sufficiently large integers $n$. It follows that $B_{r}=0$.  Furthermore, using \eqref{E:equaBL} and repeating the arguments above we can show that $B_{r-1}=\cdots=B_{1}=0$, whence the desired claim.

{\textbf{Case (II)}}. $\cL_{r+1,1}(\theta)\in k$.
In this case, we apply \cite[Thm.~3.1.1]{ABP04} to the difference equations
\[
\left(\psi:=\begin{pmatrix}
1\\
\Omega^{s_{1}+\cdots+s_{r}}\\
\Omega^{s_{2}+\cdots+s_{r}}\cL_{21}  \\
\vdots\\
 \cL_{r+1,1}
\end{pmatrix}\right)^{(-1)}=\begin{pmatrix}
1&\\
& \Phi\\
\end{pmatrix}
\begin{pmatrix}
1\\
\Omega^{s_{1}+\cdots+s_{r}}\\
\Omega^{s_{2}+\cdots+s_{r}}\cL_{21}  \\
\vdots\\
 \cL_{r+1,1}
\end{pmatrix}
\] for the $k$-linear dependence of $\left\{1,\cL_{r+1,1}(\theta) \right\}$. So there exists ${\bf{f}}=(f_{0},f_{1},\ldots,f_{r+1})\in \Mat_{1\times (r+2)}(\ok[t])$ for which ${\bf{f}}\psi=0$ and ${\bf{f}}(\theta)\psi(\theta)=0$ represents the $k$-linear dependence of $1$ and $\cL_{r+1,1}(\theta)$.
Then the arguments are similar in this case as in the previous one when putting $\tilde{\bf{f}}:=\frac{1}{f_{r+1}}{\bf{f}}$, and we omit them.  Moreover, they
exactly show that the class of $M$ is an $a$-torsion element in $\Ext_{\cF}^{1}({\mathbf{1}},M')$, where $a\in \FF_{q}[t]$ is the common denominator of $f_{0}/f_{r+1},\ldots,f_{r}/f_{r+1}$ in this case.

\subsection{Proof of Theorem~\ref{T:ThmGeneral} (a)($\Leftarrow$) and (b)}
Suppose that there exist  $a,b\in \FF_{q}[t]$ (not both zero) for which $a*M+_{B}b*N$ represents a trivial class in $\Ext_{\cF}^{1}({\mathbf{1}},M')$. Note that $a*M+_{B}b*N$ is defined by the matrix given in (\ref{E:matrix-X}). Let $\Psi'$ be the square matrix of size $r$ cut from the upper left square of $\Psi$, so that $\Psi'^{(-1)}=\Phi'\Psi'$.  Define
\[
Y:=\begin{pmatrix}
\Omega^{s_{1}+\cdots+s_{r}}& & & & & \\
\Omega^{s_{2}+\cdots+s_{r}}\cL_{21}&\Omega^{s_{2}+\cdots+s_{r}} & & & & \\
\vdots& \Omega^{s_{3}+\cdots+s_{r}} \cL_{32} &\ddots & & & \\
\vdots&\vdots &\ddots &\ddots & & \\
\Omega^{s_{r}}\cL_{r1}&\Omega^{s_{r}}\cL_{r2} & &\ddots &\Omega^{s_{r}} & \\
a\cL_{(r+1),1}+b\cL_{w,Q}&a\cL_{(r+1),2} &\cdots &\cdots &a\cL_{(r+1),r} &1 \\
\end{pmatrix},
\] and note that $Y^{(-1)}=XY$.

Since the class of $a*M+_{B}b*N$ is trivial in $\Ext_{\cF}^{1}({\mathbf{1}},M')$, there exist $\delta_{1},\ldots,\delta_{r}\in \ok[t]$ satisfying (\ref{E:deltaX}). Putting \[\delta:=\begin{pmatrix}
1& & & \\
& 1& & \\
& & \ddots& \\
\delta_{1} &\cdots &\delta_{r} & 1 \\
\end{pmatrix}\] and  $Y':=\delta Y$, we then have $Y'^{(-1)}= \left( \begin{smallmatrix} \Phi'&\\ &1\\ \end{smallmatrix} \right)Y'$. Since also
\[
\begin{pmatrix}
\Psi'&\\
&1\\
\end{pmatrix}^{(-1)}=\begin{pmatrix}
\Phi'&\\
&1\\
\end{pmatrix} \begin{pmatrix}
\Psi'&\\
&1\\
\end{pmatrix},
\]
it follows from \cite[\S 4.1.6]{P08} that there exist $\nu_{1},\ldots,\nu_{r}\in \FF_{q}(t)$ so that
\[
Y'=\begin{pmatrix}
\Psi'&\\
&1\\
\end{pmatrix} \begin{pmatrix}
I_r&\\
\nu_{1},\ldots,\nu_{r}&1\\
\end{pmatrix}.
\]
This then implies that
\begin{equation}\label{E:nudeltaL}
  \begin{split}
    \nu_{1} &= \delta_{1}\Omega^{s_{1}+\cdots+s_{r}} + \delta_{2}\Omega^{s_{2}+\cdots+s_{r}}\cL_{21} + \cdots + \delta_{r}\Omega^{s_{r}}\cL_{r1} +a\cL_{r+1,1}+b\cL_{w,Q};\\
    \nu_{2}& = \delta_{2}\Omega^{s_{2}+\cdots+s_{r}} + \delta_{3}\Omega^{s_{3}+\cdots+s_{r}}\cL_{32} + \cdots+\delta_{r}\Omega^{s_{r}}\cL_{r2}+a\cL_{r+1,2}; \\
    &\;\; \vdots \\
    \nu_{r-1} &= \delta_{r-1}\Omega^{s_{r-1}+s_{r}}+\delta_{r}\Omega^{s_{r}}\cL_{r,r-1}+a\cL_{r+1,r-1}; \\
    \nu_{r} &= \delta_{r}\Omega^{s_{r}}+a\cL_{r+1,r}.
  \end{split}
\end{equation}
We note that in fact each $\nu_{i}$ is in $\FF_{q}[t]$, since the right-hand side of each equation above is in the Tate algebra $\TT$. For any positive integer $n$, by specializing both sides of \eqref{E:nudeltaL} at $t=\theta^{q^{n}}$ and using \eqref{E:ForLji theta} together with the fact that $\Omega$ has a simple zero at $t=\theta^{q^{n}}$, we obtain that
\begin{equation}\label{E:nuzeta}
  \begin{split}
  \nu_{1}(\theta)^{q^{n}} &=  \nu_{1}(\theta^{q^{n}})=\left( a(\theta)\cL_{r+1,1}(\theta)+b(\theta)\cL_{w,Q}(\theta) \right)^{q^{n}};\\
   \nu_{2}(\theta)^{q^{n}} &= \nu_{2}(\theta^{q^{n}})=\left( a(\theta)\cL_{r+1,2}(\theta)   \right)^{q^{n}}; \\
    &\;\;\vdots \\
   \nu_{r-1}(\theta)^{q^{n}} &= \nu_{r-1}(\theta^{q^{n}})=\left( a(\theta)\cL_{r+1,r-1}(\theta)\right)^{q^{n}}; \\
   \nu_{r}(\theta)^{q^{n}} &= \nu_{r}(\theta^{q^{n}})=\left( a(\theta)\cL_{r+1,r}(\theta)\right)^{q^{n}}.
  \end{split}
\end{equation}
Since we are working in characteristic $p$, taking the $q^{n}$-th root of both sides of \eqref{E:nuzeta} shows that $\cL_{r+1,1}(\theta)(=\cL_{\fs,\fQ}(\theta))$, $\cL_{w,Q}(\theta)$ and $1$ are linearly dependent over $k$, and each of $\cL_{r+1,2}(\theta),\ldots,\cL_{r+1,r}(\theta)$ is rational in~$k$.

\section{Applications to multizeta values and multiple polylogarithms}\label{sec:applicationsMZV}
In this section, we apply Theorem~\ref{T:ThmGeneral} to establish criteria for MZV's and CMPL's at algebraic points to be Eulerian and for MZV's to be zeta-like.

\subsection{Anderson-Thakur polynomials}\label{sec:AT polynomials} Define $D_{0}=1$ and $D_{i}:=\prod_{j=0}^{i-1}\bigl(\theta^{q^{i}}-\theta^{q^{j}} \bigr)$ for $i\in \NN$. For a non-negative integer $n$, we express $n$ as
\[
n=\sum_{i=0}^{\infty} n_{i}q^{i} \quad \textnormal{($0\leq n_{i}\leq q-1$, $n_{i}=0$ for $i\gg 0$)},
\]
and we recall the definition of the Carlitz factorial,
\[
\Gamma_{n+1}:=\prod_{i=0}^{\infty} D_{i}^{n_{i}}\in A.
\]
We put $G_{0}(y):=1$ and define polynomials $G_{n}(y)\in \FF_{q}[t,y]$ for $n\in \NN$ by the product
\[
G_{n}(y)=\prod_{i=1}^{n}\left( t^{q^{n}}-y^{q^{i}} \right).
\]
Note that $G_{n+1}(y^{q})=\left( t-y^{q} \right)^{q^{n+1}} G_{n}(y)^{q}$.

For $n=0,1,2,\ldots$ we define the sequence of Anderson-Thakur polynomials $H_{n}\in A[t]$ by the generating function identity
\[
\left( 1-\sum_{i=0}^{\infty} \frac{ G_{i}(\theta) }{ D_{i}|_{\theta=t}} x^{q^{i}}  \right)^{-1}=\sum_{n=0}^{\infty} \frac{H_{n}}{\Gamma_{n+1}|_{\theta=t}} x^{n}.
\]
We note that for $0\leq n\leq q-1$ we have $H_{n}=1$. We have made a change of notation by $t\leftarrow T$, $\theta\leftarrow y$ from \cite[(3.7.1)]{AT90} in order to match better the notation in \cite{AT09}, \cite{C14}. It is shown in \cite[3.7.3]{AT90} that when one regards $H_{n}$ as a polynomial in $\theta$ over $\FF_{q}[t]$, one has $\deg_{\theta}H_{n}\leq \frac{nq}{q-1}$, whence
\begin{equation}\label{E:degHn}
\|H_{n}\|_{\infty}\leq |\theta|_{\infty}^{{nq}/(q-1)} .
\end{equation}

The crucial identity developed in \cite{AT90}, \cite{AT09} is that
\[
(H_{s-1}\Omega^{s})^{(d)}(\theta) = \frac{\Gamma_{s}S_{d}(s)}{\tilde{\pi}^{s}}, \quad \forall s \in \NN,d\in \ZZ_{\geq 0}
\]
where  $S_{d}(s)$ is the power sum
\[
S_{d}(s):=\sum_{\substack{ a\in A_{+} \\ \deg a=d }}
\frac{1}{a^{s}}\in k.
\]
It follows that if we put $\fQ=(H_{s_{1}-1},\ldots,H_{s_{r}-1})$, then by (\ref{E:degHn}) $\fQ$ satisfies the hypothesis (\ref{E:HopythesisQ}). Furthermore, specialization of the series
\[
\cL_{\fs,\fQ}:=\sum_{i_{1}>\cdots>i_{r}\geq 0} \left( \Omega^{s_{r}} H_{s_{r}-1}\right)^{(i_{r})} \cdots\left( \Omega^{s_{1}}H_{s_{1}-1}\right) ^{(i_{1})}
\]
at $t=\theta$ is equal to
\[
\Gamma_{s_{1}}\cdots\Gamma_{s_{r}}\zeta_{A}(s_{1},\ldots,s_{r})/\tilde{\pi}^{s_{1}+\cdots+s_{r}}.
\]

\subsection{A criterion for Eulerian MZV's}\label{sec:criterionEulerian}
We continue with the notation defined in the previous section.  Given $\fs=(s_{1},\ldots,s_{r})\in \NN^{r}$, we let $\fQ=(H_{s_{1}-1},\ldots,H_{s_{r}-1})$ and $M$ (resp. $M'$) be defined by $\Phi$ as in \eqref{E:Phi s} (resp. by $\Phi'$ as in \eqref{E:Phi s'}).  For this choice of $\fQ$ each series $\cL_{j\ell}$ defined in \eqref{E:Ljl} evaluates at $t=\theta$ as
\[
\cL_{j\ell}(\theta)=\Gamma_{s_{\ell}}\cdots\Gamma_{s_{j-1}}
\zeta_{A}(s_{\ell},\ldots,s_{j-1})/\tilde{\pi}^{s_{\ell}+\cdots+s_{j-1}}.
\]
These values are non-vanishing by the work of Thakur~\cite[Thm.~4]{T09a},  and so it satisfies the non-vanishing hypothesis~\eqref{E:nonzeroHyp}.  In particular,
\begin{equation}\label{E:Lr+1to1}
  \begin{split}
    \cL_{r+1,1}(\theta) &=
    \Gamma_{s_{1}}\cdots\Gamma_{s_{r}}\zeta_{A}(s_{1},\ldots,s_{r})/\tilde{\pi}^{s_{1}+\cdots+s_{r}} \\
    \cL_{r+1,2}(\theta) &=
    \Gamma_{s_{2}}\cdots\Gamma_{s_{r}}\zeta_{A}(s_{2},\ldots,s_{r})/\tilde{\pi}^{s_{2}+\cdots+s_{r}} \\
     & \;\;\vdots \\
   \cL_{r+1,r}(\theta) &= \Gamma_{s_{r}}\zeta_{A}(s_{r})/\tilde{\pi}^{s_{r}}.
  \end{split}
\end{equation}

Note that the hypotheses in Theorem~\ref{T:ThmGeneral} are satisfied and thus applying Corollary~\ref{Cor:CorThmGeneral} in this situation, we obtain the following criterion to determine when a given multizeta value is Eulerian.

\begin{theorem}\label{T:MainThm}
The multizeta value $\zeta_{A}(s_{1},\cdots,s_{r})$ is Eulerian if and only if the class of $M$ is a torsion element in the $\FF_{q}[t]$-module $\Ext_{\cF}^{1}({\mathbf{1}},M')$.
\end{theorem}

Applying Theorem~\ref{T:ThmGeneral}(b), we obtain additional information about when a given multizeta value is Eulerian.

\begin{corollary}\label{C:SimuEulerian} Let $(s_{1},\ldots,s_{r})\in \NN^{r}$ with $r\geq 2$. Suppose that $\zeta_A(s_1,\dots,s_r)$ is Eulerian.  Then the following hold.
\begin{enumerate}
\item[(a)] Each of the multizeta values
\[
\zeta_{A}(s_{2},\ldots,s_{r}),\ \zeta_{A}(s_{3},\ldots,s_{r}),\ \dots,\ \zeta_{A}(s_{r})
\]
is also Eulerian.

\item[(b)] Each $\zeta_{A}(s_{i})$ is also Eulerian for $i=1,\ldots,r$. That is, each $s_{i}$ is divisible by $q-1$.
\end{enumerate}
 \end{corollary}

\begin{proof}
Using \eqref{E:Lr+1to1}, part (a) follows from Theorem~\ref{T:ThmGeneral}(b).  For part (b), we note that $\tilde{\pi}^{n}\in k_{\infty}$ if and only if $(q-1)\mid n$. If $\zeta_{A}(s_{1},\ldots,s_{r})$ is Eulerian, then part (a) implies that so are $\zeta_{A}(s_{2},\ldots,s_{r})$, $\zeta_{A}(s_{3},\ldots,s_{r}),\ldots, \zeta_{A}(s_{r})$. Therefore, we have the divisibility relations
\[
(q-1)\mid (s_{1}+\cdots+s_{r}),\ \ldots,\ (q-1)\mid (s_{r-1}+s_{r}),\ (q-1)\mid s_{r}.
\]
It follows that each $s_{i}$ is divisible by $q-1$ for $i=1,\ldots,r$, and so each $\zeta_{A}(s_{i})$ is Eulerian by the original results of Carlitz~\cite{Carlitz}.
\end{proof}

\subsection{Application to multiple polylogarithms at algebraic points}\label{sec:CMPLs}
We now consider Carlitz multiple polylogarithms~\cite{C14}, which are generalizations of Carlitz polylogarithms~\cite{Carlitz, AT90}.  We state a criterion to determine when a nonzero Carlitz multiple polylogarithm at an algebraic point is Eulerian and show how to deduce it from Corollary~\ref{Cor:CorThmGeneral}.

Define $L_{0}:=1$ and $L_i :=(\theta-\theta^{q})\cdots(\theta-\theta^{q^{i}})$ for $i\in \NN$. Given an $r$-tuple of positive integers $\fs=(s_{1},\ldots,s_{r})\in \NN^{r}$, its associated Carlitz multiple polylogarithm, abbreviated CMPL, is defined by
\begin{equation}\label{E:DefCMPL}
 \Li_{\fs}(z_{1},\ldots,z_{r}):=\sum_{i_{1}>\cdots> i_{r}\geq 0} \frac{z_{1}^{q^{i_{1}}}\cdots z_{r}^{q^{i_{r}}}} {L_{i_{1}}^{s_{1}}\cdots L_{i_{r}}^{s_{r}}}.
\end{equation}
We further denote by
\[
\mathbb{D}_{\mathfrak{s}}:= \{\bu=(u_{1},\ldots,u_{r})\in \CC_{\infty}^{r} \mid \textnormal{$\Li_{\fs}(\bu)$ converges} \},
\]
the convergence domain of $\Li_{\fs}$.  We can describe $\mathbb{D}_{\mathfrak{s}}$ as
\[
\mathbb{D}_{\fs} = \left\{ \bu \in \CC_{\infty}^r \Bigm|
  \bigl\lvert u_{1}/(\theta^{\frac{qs_{1}}{q-1}} ) \bigr\rvert_{\infty}^{q^{i_{1}}}
  \cdots \bigl\lvert u_{r}/ (\theta^{\frac{qs_{r}}{q-1}}) \bigr\rvert_{\infty}^{q^{i_{r}}}
  \rightarrow 0 \textnormal{\ as\ } 0\leq  i_{r}< \cdots< i_{1}\rightarrow \infty   \right\}.
\]
The weight of $\Li_{\fs}$ is defined to be $\wt(\fs):=\sum_{i=1}^{r}s_{i}$ and its depth is defined to be $r$.  We do not know whether $\Li_{\fs}(\bu)$ is non-vanishing for arbitrary $\bu\in\mathbb{D}_{\fs}$, but it is non-vanishing when $\bu$ lies in the smaller domain
\[
\mathbb{D}_{\fs}':= \left\{ (u_{1},\ldots,u_{r})\in \CC_{\infty}^{r} \mid \lvert u_{i}\rvert_{\infty} < q^{\frac{s_{i}q}{q-1}} \textnormal{for $i=1,\ldots,r$} \right\}
\]
(see~\cite[Remark~5.1.4]{C14}).

As a generalization of the work of Anderson and Thakur~\cite{AT90}, the first author of the present paper showed that for any $\fs=(s_{1},\ldots,s_{r})\in \NN^{r}$, $\zeta_{A}(s_{1},\ldots,s_{r})$ is a $k$-linear combination of $\Li_{\fs}$ at some integral points in $A^{r}\cap \mathbb{D}_{\fs}$~\cite[Thm.~5.5.2]{C14}. From now on, we fix an $r$-tuple $\fs=(s_{1},\ldots,s_{r})\in \NN^{r}$ and $\bu=(u_{1},\ldots,u_{r})\in (\bar{k}^{\times})^{r}\cap \mathbb{D}_{\fs}$, and suppose that $\Li_{\fs}(\bu)\neq 0$. Following the terminology for multizeta values, we shall call $\Li_{\fs}(\bu)$ \emph{Eulerian} if the ratio $\Li_{\fs}(\bu)/\tilde{\pi}^{\wt(\fs)}$ lies in $k$. Note that it is shown in \cite{C14} that $\Li_{\fs}(\bu)/\tilde{\pi}^{\wt(\fs)}\in \bar{k}$ if and only if $\Li_{\fs}(\bu)/\tilde{\pi}^{\wt(\fs)}\in k$.

We put $\fQ:=\bu=(u_{1},\ldots,u_{r})$ and let $M$ (resp. $M'$) be the Frobenius module defined by $\Phi$ as in \eqref{E:Phi s} (resp. $\Phi'$ as in \eqref{E:Phi s'}). For $1\leq i < j\leq r$, we put
\[
\Li_{\fs_{ij}}(\bu_{ij}):=\Li_{(s_{i},\dots,s_{j})}(u_{i},\dots,u_{j}).
\] Note that in this setting we have $\cL_{\fs,\fQ}(\theta)=\Li_{\fs}(\bu)/\tilde{\pi}^{\wt(\fs)}$, and hence $\Li_{\fs}(\bu)$ is Eulerian if and only if $\cL_{\fs,\fQ}(\theta)\in k$.  Applying Corollary~\ref{Cor:CorThmGeneral} and  Theorem~\ref{T:ThmGeneral}(b), we obtain the following result.

\begin{theorem}\label{T:MainThmCMPL}
 For any $\fs=(s_{1},\ldots,s_{r})\in \NN^{r}$ and $\bu=(u_{1},\ldots,u_{r})\in (\bar{k}^{\times})^{r}\cap \mathbb{D}_{\fs}$, we let $M$ and $M'$ be defined as above. Suppose that $\Li_{\fs_{ij}}(\bu_{ij})$ is nonzero for all $1\leq i <j\leq r$. Then we have
\begin{enumerate}
\item[(a)]  $\Li_{\fs}(\bu)$ is Eulerian if and only if $M$ represents a torsion element in the $\FF_{q}[t]$-module $\Ext_{\cF}^{1}({\mathbf{1}},M')$.
\item[(b)]  If $\Li_{\fs}(\bu)$ is Eulerian, then each of
$\Li_{(s_{2},\ldots,s_{r})}(u_{2},\ldots,u_{r}),\ldots,\Li_{s_{r}}(u_{r})$
is also Eulerian.
\end{enumerate}
\end{theorem}

\begin{remark}
The non-vanishing hypothesis of $\Li_{\fs_{ij}}(\bu_{ij})$ is non-empty.  Indeed, for those points $\bu\in \mathbb{D}_{\fs}'$ we have that $\Li_{\fs_{ij}}(\bu_{ij})\neq 0$ for all $1\leq i < j\leq r$. See~\cite[Rem.~5.1.5]{C14}.
\end{remark}

\subsection{Applications to zeta-like multizeta values}\label{sec:applications}
In this section we apply Corollary~\ref{Cor:CorThmGeneral} to confirm a conjecture of Lara Rodr\'{i}guez and Thakur.  As defined by Thakur a multizeta value $\zeta_{A}(s_{1},\ldots,s_{r})$ with weight $w:=\sum_{i=1}^{r}s_{i}$ is called \emph{zeta-like} if the ratio $\zeta_{A}(s_{1},\ldots,s_{r})/\zeta_{A}(w)$ is algebraic over $k$. Note that if $\zeta_{A}(s_{1},\ldots,s_{r})$ is zeta-like, then by \cite[Cor.~2.3.3]{C14} the ratio $\zeta_{A}(s_{1},\ldots,s_{r})/\zeta_{A}(w)$ is actually in $k$.

\subsubsection{Zeta-like MZV's} For a multizeta value $\zeta_{A}(s_{1},\ldots,s_{r})$, Lara Rodr\'{i}guez and Thakur~\cite{LRT13} conjectured the following assertion.
\begin{conjecture}
If $\zeta_{A}(s_{1},\ldots,s_{r})$ is zeta-like, then $\zeta_{A}(s_{2},\ldots,s_{r})$ is Eulerian.
\end{conjecture}

In what follows, we confirm this conjecture stated as Corollary~\ref{C:zeta-like}, which is a consequence of the following theorem. For more  conjectures concerning zeta-like MZV's, we refer the reader to \cite{LRT13}.

\subsubsection{The setting and results}
Given $\fs=(s_{1},\ldots,s_{r})\in \NN^{r}$, put $\fQ=(H_{s_{1}-1},\ldots,H_{s_{r}-1})$ and $Q:=H_{w-1}$, where $\left\{ H_{n}\right\}$ be the Anderson-Thakur polynomials given in \S\ref{sec:AT polynomials} and $w:=\sum_{i=1}^{r}s_{i}$. Note that
\[\cL_{w,Q}(\theta)=\frac{\Gamma_{w}\zeta_{A}(w)}{\tilde{\pi}^{w}}, \]
and so by \eqref{E:degHn} the conditions of Theorem~\ref{T:ThmGeneral} are satisfied.

Let $M\in \cF$ (resp. $N\in \cF$) be defined by $\Phi$ given in \eqref{E:Phi s} (resp. by the matrix given in \eqref{E:PhiN}). Then applying Theorem~\ref{T:ThmGeneral} we have the following criterion.
\begin{theorem}\label{T:Zeta-like}
Let $\fs=(s_{1},\ldots,s_{r})\in \NN^{r}$ with $r\geq 2$, and let notation be given as above. Then we have that $\left\{ \zeta_{A}(s_{1},\ldots,s_{r}),\zeta_{A}(w),\tilde{\pi}^{w}  \right\}$ are linearly dependent over $k$ if and only if the classes of $M$ and $N$ are $\FF_{q}[t]$-linearly dependent in $\Ext_{\cF}^{1}({\bf{1}},M')$.
\end{theorem}

\begin{corollary}\label{C:zeta-like}
If $\zeta_{A}(s_{1},\ldots,s_{r})$ is zeta-like, then each of
\[
\zeta_{A}(s_{2},\ldots,s_{r}),\zeta_{A}(s_{3},\ldots,s_{r})\ldots,\zeta_{A}(s_{r})
\]
is also Eulerian. In particular, each of $s_{2},\ldots,s_{r}$ is divisible by $q-1$.
\end{corollary}

\begin{proof}
Note that
\[\cL_{\fs,\fQ}(\theta)=\Gamma_{s_{1}}\cdots\Gamma_{s_{r}}\zeta_{A}(s_{1},\ldots,s_{r})/\tilde{\pi}^{w},\hbox{ }\cL_{w,Q}(\theta)=\Gamma_{w}\zeta_{A}(w)/\tilde{\pi}^{w} \] and for each $2\leq\ell\leq r$,
\[\cL_{r+1,\ell}(\theta)=\Gamma_{s_{\ell}}\cdots\Gamma_{s_{r}}\zeta_{A}(s_{\ell},\ldots,s_{r})/\tilde{\pi}^{s_{\ell}+\cdots+s_{r}}   .\]
Since by hypothesis $\zeta_{A}(s_{1},\ldots,s_{r})$ is zeta-like, the result follows from Theorem~\ref{T:ThmGeneral}(b).
\end{proof}

\begin{remark}
In this remark, we give an algebraic illustration for the result in the corollary above. For an $r$-tuple $\fs=(s_{1},\ldots,s_{r})\in \NN^{r}$ with $r\geq 2$ and $w:=\sum_{i=1}^{r}s_{i}$, we continue the notation $M_{\fs},N_{\fs},M_{\fs}'$ above but we add the subscript $\fs$ to emphasize that these objects are associated to $\fs$. We put $\fs':=(s_{2},\ldots,s_{r})$ and let $M_{\fs'},M_{\fs'}'$ be those objects above associated to $\fs'$. It is shown in the proof of Theorem~\ref{T:EffectiveCriterion} that one has the exact sequence of $\FF_{q}[t]$-modules (via Theorems~\ref{T:Ext1} and \ref{T:t-module}):
\[
\xymatrix{
0\ar[r] & \Ext_{\cF}^{1}\left( {\bf{1}},C^{\otimes w} \right) \ar[r]  & \Ext_{\cF}^{1}\left( {\bf{1}},M_{\fs}' \right) \ar@{->>}[r]^{\pi}  & \Ext_{\cF}^{1}\left( {\bf{1}},M_{\fs'}' \right)  \ar[r]  & 0,
}
\] and note that $N_{\fs}\in  \Ext_{\cF}^{1}\left( {\bf{1}},C^{\otimes w} \right)$ and $\pi (M_{\fs})=M_{\fs'}$.  Suppose that $\zeta_{A}(\fs)$ is zeta-like. Then by Remark~\ref{Rem:Coeff:a} there exist $a,b\in \FF_{q}[t]$ with $a\neq 0$ so that $a*M_{\fs}+b* N_{\fs}$ represents a trivial class in  $\Ext_{\cF}^{1}\left( {\bf{1}},M_{\fs}' \right)$. It follows that $\pi\left( a*M_{\fs}+b* N_{\fs} \right)= a* M_{\fs'}$ represents a trivial class in $\Ext_{\cF}^{1}\left( {\bf{1}},M_{\fs'}' \right) $, and hence by Theorem~\ref{T:MainThm} $\zeta_{A}(\fs')$ is Eulerian.

\end{remark}

\section{Operating on $t$-modules}\label{sec:t-modules}

The purpose of this section is to reformulate our criteria via $t$-modules.
\subsection{The structure of rational torsion points of $\bC^{\otimes n}$}

\subsubsection{Definition of $t$-modules}
We first review the definition of $t$-modules~\cite{A86}. Let $\tau=(x\mapsto x^{q}):\CC_{\infty}\rightarrow \CC_{\infty}$ be the Frobenius $q$-th power operator and let $\CC_{\infty}[\tau]$ be the twisted polynomial ring in $\tau$ over $\CC_{\infty}$ subject to the relation $\tau \alpha=\alpha^{q} \tau$ for $\alpha\in \CC_{\infty}$. For a positive integer $d$, a $d$-dimensional $t$-module is a pair $(E,\phi)$, where $E$ is the $d$-dimensional algebraic group $\GG_{a}^{d}$ and $\phi$ is an $\FF_{q}$-linear ring homomorphism
\[
\phi: \FF_{q}[t]\rightarrow \Mat_{d}(\CC_{\infty}[\tau])
\]
so that when we write $\phi_{t}=\alpha_{0}+\sum_{i}\alpha_{i}\tau^{i}$ with $\alpha_{i}\in \Mat_{d}(\CC_{\infty})$, $\alpha_{0}-\theta I_{d}$ is a nilpotent matrix. In this way, $E(\CC_{\infty})$ is equipped with an $\FF_{q}[t]$-module structure via the map $\phi$. For a subring $R\subset \CC_{\infty}$ containing $A$, we say that the $t$-module $E$ is defined over $R$ if $\alpha_{i}$ lies in $\Mat_{d}(A)$ for all $i\geq 0$.

For any $d$-dimensional  $t$-module $(E,\phi)$, Anderson~\cite{A86} showed that one has the $\FF_{q}$-linear function $\exp_{E}:\CC_{\infty}^{d}\rightarrow \CC_{\infty}^{d}$ satisfying that for ${\bf{z}}=(z_{1},\ldots,z_{d})^{\tr}$ and any $a\in \FF_{q}[t]$,
\begin{enumerate}
\item[$\bullet$] $\exp_{E}(\bz)\equiv \bz$ (mod degree $q$)
\item[$\bullet$] $\exp_{E}(\partial\phi_{a}(\bz) )= \phi_{a}\left(\exp_{E}(\bz) \right) $,
\end{enumerate}
where $\partial \phi_{a}$ is the differential of the morphism $\phi_{a}$ at the identity element of $E$. If $\exp_{E}$ is surjective, then $E$ is called uniformizable.

\subsubsection{Anderson-Thakur special points}\label{sec:AT special points}
For a positive integer $n$, the $n$-th tensor power of the Carlitz $\FF_{q}[t]$-module denoted by $\bC^{\otimes n}$ is an $n$-dimensional $t$-module defined over $A$ together with the $\FF_{q}$-linear ring homomorphism \[ [\cdot]_{n}:\FF_{q}[t]\rightarrow \Mat_{n}(\CC_{\infty}[\tau])  \]
given by \[[t]_{n}=\theta I_{n}+N_{n}+E_{n} \tau ,\]
where
\[N_{n}:=\begin{pmatrix}
0& 1& \cdots &0\\
\vdots & \ddots&\ddots & \vdots\\
\vdots &  &\ddots &1\\
0&\cdots &\cdots & 0
\end{pmatrix}, \quad
E_{n}:= \begin{pmatrix}
0& \cdots&\cdots  &0 \\
\vdots&   & &\vdots\\
\vdots& & &\vdots \\
1& \cdots& \cdots &0
\end{pmatrix}.
\]
Note that for $n=1$, the definition above is the Carlitz $\FF_{q}[t]$-module $\bC$. It is shown in \cite[Cor.~2.5.8]{AT90} that $\bC^{\otimes n}$ is uniformizable and the kernel of $\exp_{\bC^{\otimes n}}$ is a rank one $\FF_{q}[t]$-module (via the $\partial [a]_{n}$-action) with a generator of the form
\[ \lambda_{n}=\begin{pmatrix} * \\ \vdots  \\ \tilde{\pi}^{n} \end{pmatrix}
\in \CC_{\infty}^{n}. \]

To find the connection with $\zeta_{A}(n)$, Anderson and Thakur defined the following special points (see \cite[(3.8.2)]{AT90}).

\begin{definition}\label{Def:special points}
For each positive integer $n$, we let $H_{n-1}\in A[t]$ be the Anderson-Thakur polynomial in \S\ref{sec:AT polynomials}. We write $H_{n-1}=\sum_{i\geq 0}h_{ni}\theta^{i}$ with $h_{ni}\in \FF_{q}[t]$. Then we define
\[ Z_{n}:=\sum_{i\geq 0}[h_{ni}]_{n} \begin{pmatrix} 0 \\ \vdots \\0 \\ \theta^{i} \end{pmatrix}
 \in \bC^{\otimes n}(A)\] and call it an Anderson-Thakur special point.
\end{definition}

It is shown in \cite[Thm.~3.8.3]{AT90} that there exists a vector of the form
\[z_{n}= \left(
                                                      \begin{array}{c}
                                                        * \\
                                                        \vdots \\
                                                        * \\
                                                        \Gamma_{n}\zeta_{A}(n) \\
                                                      \end{array}
                                                    \right)
\]
so that
\[
\exp_{\bC^{\otimes n}}(z_{n})=Z_{n}.
\]

\begin{remark}\label{Rem:torsion Zn}
For a positive integer $n$ {\it{even}}, we put
\[
a:=\frac{\Gamma_{n+1}}{\Gamma_{n}}\den\BC(n)|_{\theta = t},
\]
where $\den\BC(n)$ denotes the denominator of the $n$-th Bernoulli-Carlitz number.
Then by the formula \eqref{E:CarlitzFormula} of Carlitz, the property that $\Ker \exp_{\bC^{\otimes n}}$ is a rank one $\FF_{q}[t]$-module, the functional equation $\exp_{\bC^{\otimes n}}(\partial [a]_{n}\bz)=[a]_{n}\left(\exp_{\bC^{\otimes n}}(\bz) \right)$ and \cite[Thm.~2.3]{Yu91}, we see that $Z_{n}$ is an $a$-torsion point in $\bC^{\otimes n}(A)$.
\end{remark}

\subsubsection{The structure of $\bC^{\otimes n}(k)_{\rm{tor}}$}
For any nonzero polynomial $f \in \FF_{q}[t]$, denote by $\bC^{\otimes n}[f]$ the set of $f$-torsion elements:
\[
\bC^{\otimes n}[f]:=\left\{\bx\in \bC^{\otimes n}(\ok)= \ok^{n}\mid [f]_{n}(\bx)={\bf{0}}  \right\}.
\]
We further define the set of rational torsion points of $\bC^{\otimes n}$:
\[
\bC^{\otimes n}(k)_{\mathrm{tor}}:=\left\{\bx\in \bC^{\otimes n}(k)= k^{n}\mid [f]_{n}(\bx)={\bf{0}}\hbox{ for some nonzero }f\in \FF_{q}[t]  \right\}.
\]
Note that in \cite[Prop.~1.11.2]{AT90} Anderson and Thakur showed that $\bC^{\otimes n}(k)_{\rm{tor}}$ is trivial if $n$ is {\textit{odd}} (ie., $(q-1)\nmid n$). The following result is the structure of $\bC^{\otimes n}(k)_{\rm{tor}}$ when $n$ is {\textit{even}}, and we thank Y.-L. Kuan for providing us a proof of the following lemma.
\begin{lemma}\label{L:RationalTorsions} Let $n$ be a positive integer divisible by $q-1$. We decompose $n = p^{\ell} n_1\left(q^h-1\right)$ where $p\nmid n_1$ and $h$ is the greatest integer such that $(q^h-1) \mid n$. Then
\[
\bC^{\otimes n}(k)_{\rm{tor}} = \prod_{\deg P \mid h}\bC^{\otimes n}[P^{p^{\ell}}],
\]
where the product runs through all monic irreducible polynomials $P \in \FF_{q}[t]$ with $\deg P \mid h$. In particular, the Fitting ideal of the finite $\FF_{q}[t]$-module $\bC^{\otimes n}(k)_{\rm{tor}}$ is generated by $(t^{q^{h}}-t)^{p^{\ell}}$.
\end{lemma}

\begin{proof}
For any nonzero $f \in \FF_{q}[t]$, let $\left(\FF_{q}[t]/f\right)^{\times n}$ be the group of the $n$-th powers of elements of $\left(\FF_{q}[t]/f\right)^{\times}$. By \cite[Prop.~1.11.1]{AT90} one knows that the Galois group ${\rm{Gal}}\left(k\left(\bC^{\otimes n}[f]\right)/k\right)$ is isomorphic to $\left(\FF_{q}[t]/f\right)^{\times n}$. It follows that $\bC^{\otimes n}[f]\subseteq \bC^{\otimes n}(k)$ if and only if the group $\left(\FF_{q}[t]/f\right)^{\times n}$ is trivial.

Note that the group $\left(\FF_{q}[t]/P\right)^{\times n}$ is non-trivial if $\deg P \nmid h$. Thus, to prove the lemma it suffices to show that the group $(\FF_{q}[t]/P^{p^{\ell}})^{\times n}$ is trivial and $(\FF_{q}[t]/P^{p^{\ell+ 1}})^{\times n}$ is non-trivial for any monic irreducible polynomial $P$ with $\deg P \mid h$. Suppose that $\deg P \mid h$. Let $(\FF_{q}[t]/P^{p^{\ell}})^{(1)}$ be the kernel of the natural map from $(\FF_{q}[t]/P^{p^{\ell}})^{\times}$ to $(\FF_{q}[t]/P)^\times$. Then the group $(\FF_{q}[t]/P^{p^{\ell}})^{\times}$ is isomorphic to the direct product of $(\FF_{q}[t]/P)^\times$ and $(\FF_{q}[t]/P^{p^{\ell}})^{(1)}$. Since $\deg P \mid h$, the group $(\FF_{q}[t]/P)^{\times n}$ is trivial. Note that every element in $(\FF_{q}[t]/P^{p^{\ell}})^{(1)}$ can be represented by a polynomial of the form $a = 1 + bP$ with $b \in \FF_{q}[t]$. Then
\[
a^n = \left((1 + bP)^{p^{\ell}}\right)^{n_1(q^h - 1)} = \left(1 + b^{p^{\ell}}P^{p^{\ell}}\right)^{n_1(q^h - 1)} \equiv 1 \pmod {P^{p^{\ell}}},
\] and hence the group $(\FF_{q}[t]/P^{p^{\ell}})^{\times n}$ is trivial. On the other hand, we now consider $a = 1+P$. Since $(\FF_{q}[t]/P^{p^{\ell+1}})^{(1)}$ is a $p$-group, $a^n \neq 1$ in $(A/P^{p^{\ell+1}})^{(1)}$ if and only if $a^{p^{\ell}} \neq 1$ in $(\FF_{q}[t]/P^{p^{\ell+1}})^{(1)}$.
As it is clear that $a^{p^{\ell}} \not\equiv 1 \pmod {P^{p^{\ell + 1}}}$, the proof of the desired result is completed.
\end{proof}

\subsection{The $\Ext^{1}$-modules and $t$-modules}\label{sec:Ext^1} In this subsection, we will give an identification between certain $\Ext^{1}$-modules and $t$-modules. The key ingredient and ideas exhibited here are not new; actually they are due to G.~Anderson, who shared his unpublished notes with the authors.  Elements of these constructions are also presented in the works of Hartl and Pink~\cite{HP04} and Taelman~\cite{Taelman10}.  In what follows, we fix two $r$-tuples $\fs=(s_{1},\ldots,s_{r})\in \NN^{r}$ and $\fQ\in \ok[t]^{r}$ satisfying \eqref{E:HopythesisQ}. Associated these two $r$-tuples, we let $M$ (resp. $M'$) be the Frobenius module defined by $\Phi$ as in \eqref{E:Phi s} (resp. $\Phi'$ as in \eqref{E:Phi s'}).

\begin{theorem}\label{T:Ext1}
Let $\left\{m_{1},\ldots,m_{r}\right\}$ be a $\ok[t]$-basis of $M'$ on which the $\sigma$-action is presented by the matrix $\Phi'$. Let $M\in \Ext_{\cF}^{1}({\mathbf{1}},M')$ be defined by the matrix
\[ \begin{pmatrix}
\Phi'& 0\\
f_{1},\ldots,f_{r}&1\\
\end{pmatrix}.  \] Then the map
\[ \mu:=\left( M\mapsto f_{1}m_{1}+\cdots+f_{r}m_{r} \right):\Ext_{\cF}^{1}({\mathbf{1}},M') \rightarrow M'/(\sigma-1)M'   \]
is an isomorphism of $\FF_{q}[t]$-modules.
\end{theorem}
\begin{proof} We first show that the map $\mu$ is well-defined. Suppose that $M$ is trivial in $\Ext_{\cF}^{1}({\mathbf{1}},M')$. Equivalently, there exists
$u_{1},\ldots,u_{r}\in \ok[t]$ so that
\[ \begin{pmatrix}
I_{r}\\
u_{1},\ldots,u_{r}&1\\
\end{pmatrix}^{(-1)}  \begin{pmatrix}
\Phi'& \\
f_{1},\ldots,f_{r}&1\\
\end{pmatrix}= \begin{pmatrix}
\Phi'& \\
 &1\\
\end{pmatrix} \begin{pmatrix}
I_{r}\\
u_{1},\ldots,u_{r}&1\\
\end{pmatrix}
.\] The equation above is equivalent to
\[ (f_{1},\ldots,f_{r})(m_{1},\ldots,m_{r})^{{\tr}}=\left( (u_{1},\ldots,u_{r})-(u_{1}^{(-1)},\ldots,u_{r}^{(-1)}) \Phi' \right) (m_{1},\ldots,m_{r})^{{\tr}},  \]
which is equivalent to
\[ (f_{1},\ldots,f_{r})(m_{1},\ldots,m_{r})^{{\tr}}=(\sigma-1) \left(  (-u_{1},\ldots,-u_{r})(m_{1},\ldots,m_{r})^{{\tr}}   \right).     \]
So we have shown that $\mu$ is well-defined and also that $\mu$ is one to one.
It is clear that $\mu$ is $\FF_{q}[t]$-linear and also surjective.
\end{proof}

Now let us consider the $n$-th tensor power of the Carlitz module. The Frobenius module associated to $\bC^{\otimes n}$ is the $n$-th tensor power of the Carlitz motive denoted by $C^{\otimes n} := \ok[t]$, on which $\sigma$ acts by
\[
   \sigma(f) = (t-\theta)^n f^{(-1)}, \quad f \in C^{\otimes n}.
\]
As a $\ok[\sigma]$-module, $C^{\otimes n}$ is free of rank $n$ with basis $\left\{(t-\theta)^{n-1}, \ldots, t-\theta, 1\right\}$. From this observation, it is not hard to check that the Frobenius module $M'$ fixed as above is a free left $\ok[\sigma]$-module of rank $d:=(s_{1}+\cdots+s_{r})+(s_{2}+\cdots+s_{r})+\cdots+s_{r}$, and
 \[\left\{(t-\theta)^{s_{1}+\cdots+s_{r}-1}m_{1},\cdots,(t-\theta)m_{1},m_{1},\ldots,(t-\theta)^{s_{r}-1}m_{r},\ldots,(t-\theta)m_{r},m_{r} \right\}\]
 is a $\ok[\sigma]$-basis of $M'$. We further observe that $(t-\theta)^{N}M'/\sigma M'=(0)$ for $N\gg 0$ and hence $M'$ is an Anderson $t$-motive in the sense of \cite{P08}, which is called a dual $t$-motive in \cite{ABP04}.

For such $M'$, we can identity $M'/(\sigma-1)M'$ with the direct sum of $d$ copies of $\ok$ as follows.  Fixing a $\ok[\sigma]$-basis $\nu_1, \dots, \nu_d$ of $M'$ given as above, we can express any $m \in M'$ as
\[
  m = \sum_{i=1}^d u_i \nu_i, \quad u_i \in \ok[\sigma],
\]
and then we define $\Delta : M' \to \Mat_{d\times 1}(\ok)$ by
\begin{equation}\label{E:Delta}
  \Delta(m) := \begin{pmatrix}
  \delta(u_1) \\ \vdots \\ \delta(u_d)
  \end{pmatrix},
\end{equation}
where
\[
  \delta \biggl( \sum_{i}c_{i}\sigma^{i}= \sum_{i}  \sigma^i c_{i}^{q^{i}} \biggr) = \sum_{i} c_i^{q^i}.
\]
It follows that $\Delta$ is a morphism of $\FF_q$-vector spaces with kernel $(\sigma-1)M'$.  We note that if $(a_1, \dots, a_d)^{\tr} \in \Mat_{d \times 1}(\ok)$, then there is a natural lift to $M'$, since
\[
  \Delta(a_1 \nu_1 + \cdots +a_d \nu_d) = \begin{pmatrix} a_1 \\ \vdots \\ a_d \end{pmatrix}.
\]
As $t(\sigma -1)M' \subseteq (\sigma-1)M'$, the map $\Delta$ induces an $\FF_q[t]$-module structure on $\Mat_{d\times 1}(\ok)$. We denote by $(E',\rho)$ the $t$-module defined over $\ok$ with $E'(\ok)$ identified with $\Mat_{d\times 1}(\ok)$, on which the $\FF_{q}[t]$-module structure is given by
\[ \rho: \FF_{q}[t]\rightarrow \Mat_d(\ok[\tau]) \]
so that
\[  \Delta(t(a_1 \nu_1 + \cdots +a_d \nu_d))=\rho_t  \begin{pmatrix} a_1 \\ \vdots \\ a_d \end{pmatrix}.  \]

For example, we consider $C^{\otimes n}$. As a $\ok[\sigma]$-module, $C^{\otimes n}$ is free of rank $n$ with basis $\left\{ (t-\theta)^{n-1}, \ldots, t-\theta, 1\right\}$.  We let
\[
\Delta_n : C^{\otimes n} \to \Mat_{n\times 1}(\ok)
\]
be defined as above with respect to this basis.  For $(a_1, \dots, a_n)^{\tr} \in \Mat_{n\times 1}(\ok)$, we let
\[
  f = a_1 (t-\theta)^{n-1} + \cdots + a_{n-1}(t-\theta) + a_n,
\]
so that $\Delta_n(f) = (a_1, \dots, a_n)^{\tr}$.  Now
\[
  tf = (\theta a_1 + a_2) (t-\theta)^{n-1} + \cdots (\theta a_{n-1} + a_{n})(t-\theta) + (a_1^q + \theta a_n) + (\sigma-1) \bigl(a_{1}^q \bigr),
\]
and thus multiplication by $t$ on $\Mat_{n \times 1}(\ok)$ is given by
\[
  t \cdot \begin{pmatrix} a_{1} \\ \vdots \\ a_n \end{pmatrix} = \Delta_n(tf) =
  [t]_{n} \begin{pmatrix} a_{1} \\ \vdots \\ a_n \end{pmatrix}.
\]
In this way we identify  $C^{\otimes n}/(\sigma-1)C^{\otimes n}$ and $\bC^{\otimes n}(\ok)$ as $\FF_q[t]$-modules.  This identification between an abelian $t$-module over $\ok$ and the quotient of its associated dual $t$-motive modulo $\sigma-1$ is due entirely to Anderson.  One sees it in \cite[\S 4]{ABP04} (see especially  the functor ``$f \pmod{\sigma -1}$'' in \S 4.1).  See also \cite[\S 4.6]{BP}, \cite{CP11}, \cite{CP12} for other instances of this phenomenon for Drinfeld modules.

To summarize,  we have the following result.

\begin{theorem}[Anderson] \label{T:t-module}
Let $M'$ be the Frobenius module defined by the matrix $\Phi'$ in \eqref{E:Phi s'}. Let $(E',\rho)$ be the $t$-module with $E'(\ok)$ identified with $\Mat_{d\times 1}(\ok)$, which is equipped with the $\FF_{q}[t]$-module structure
via $\rho:\FF_{q}[t]\rightarrow  \Mat_{d}(\ok[\tau])$ through the map $\Delta$ as above. Then we have the following isomorphism of $\FF_{q}[t]$-modules
\[
  M'/(\sigma-1)M'\cong E'(\ok),
\]
and in fact $E'$ is the $t$-module associated to the Anderson dual $t$-motive $M'$.
\end{theorem}

\begin{remark}
Combining the two theorems above we have that $\Ext_{\cF}^{1}({\mathbf{1}},M')\cong E'(\ok)$ as $\FF_{q}[t]$-modules.  Examples of this type of isomorphism were also studied by Ramachandran and the second author~\cite{PR03} for extensions of tensor powers of the Carlitz module.  See also \cite[p.~529]{Sinha97}, \cite{Taelman10}. We further mention that in fact $M'$ is  a rigid analytically trivial Anderson $t$-motive as we have $\Psi'^{(-1)}=\Phi' \Psi'$ and so the corresponding $t$-module $E'$ is uniformizable.
\end{remark}

\subsection{Reformulation of the criteria via $t$-modules}

\begin{proposition}\label{P:PropI}
Let $n$ be a positive integer. Then for any nonzero polynomial $f\in A[t]$, we have  $\Delta_{n}(f)\in \bC^{\otimes n}(A)$. Equivalently, there exist $a_{1},\ldots,a_{n}\in A$ and $g\in C^{\otimes n}$ so that
\[ f= a_{1}(t-\theta)^{n-1}+\cdots+a_{n}+ (\sigma-1)g.  \]
\end{proposition}

\begin{proof}
We write $f=\sum f_{i}\theta^{i}$ with $f_{i}\in \FF_{q}[t]$. Then via the $\FF_{q}[t]$-linear map $\Delta_{n}$ we have
\[ \Delta_{n}(\sum_{i}f_{i}\theta^{i})=\sum_{i}[f_{i}]_{n}\Delta_{n}(\theta^{i})=\sum_{i}[f_{i}]_{n} \begin{pmatrix} 0\\ \vdots\\ \theta^{i}
\end{pmatrix} \in \bC^{\otimes n}(A). \]
\end{proof}

\begin{proposition}\label{P:PropII}
Let $M'$ be the Frobenius module defined by the matrix $\Phi'$ in \eqref{E:Phi s'} with a $\ok[t]$-basis $m_{1},\ldots,m_{r}$. Let $\left\{\nu_{1},\ldots,\nu_{d} \right\}$ be the $\ok[\sigma]$-basis of $M'$ given by
 \[\left\{(t-\theta)^{s_{1}+\cdots+s_{r}-1}m_{1},\cdots,(t-\theta)m_{1},m_{1},\ldots,(t-\theta)^{s_{r}-1}m_{r},\ldots,(t-\theta)m_{r},m_{r}\right\} .\] Let $\Xi$ be the set consisting of all elements in $M'$ of the form $\sum_{i=1}^{d}e_{i}\nu_{i}$, where $e_{j}=\sum_{n}\sigma^{n} u_{nj}$ with each $u_{nj}\in A$. Then for any nonzero $f\in A[t]$ and any $1\leq \ell\leq r$, we have that $fm_{\ell} \in \Xi$.
\end{proposition}
\begin{proof} We first prove the case when $\ell=1$. We divide $f$ by $(t-\theta)^{s_{1}+\cdots+s_{r}}$ and write \[f=g_{1}(t-\theta)^{s_{1}+\cdots+s_{r}}+\gamma_{1},\] where $g_{1},\gamma_{1}\in A[t]$ with $\deg_{t}\gamma_{1} < s_{1}+\cdots+s_{r}$. So $fm_{1}=g_{1}\sigma m_{1}+\gamma_{1} m_{1}=\sigma g_{1}^{(1)} m_{1}+\gamma_{1} m_{1}$. Note that by expanding $\gamma$ in terms of powers of $(t-\theta)$ we see that $\gamma_{1} m_{1}$ is an $A$-linear combination of $\left\{\nu_{1},\ldots,\nu_{s_{1}+\cdots+s_{r}} \right\}$.

Next we divide $g_{1}^{(1)}\in A[t]$ by $(t-\theta)^{s_{1}+\cdots+s_{r}}$ and write \[g_{1}^{(1)}=g_{2}(t-\theta)^{s_{1}+\cdots+s_{r}}+\gamma_{2},\] where $g_{2},\gamma_{2}\in A[t]$ with $\deg_{t}\gamma_{2} < s_{1}+\cdots+s_{r}$. So \[\sigma g_{1}^{(1)} m_{1}=\sigma\left(g_{2}(t-\theta)^{s_{1}+\cdots+s_{r}}+\gamma_{2}\right) m_{1}=\sigma^{2}g_{2}^{(1)}m_{1}+\sigma\gamma_{2}m_{2}.\]
By expanding $\gamma_{2}$ in terms of powers of $(t-\theta)$ we see that $\sigma\gamma_{2}m_{2}\in \Xi$. By dividing $g_{2}^{(1)}$ by $(t-\theta)^{s_{1}+\cdots+s_{r}}$ and continuing the procedure as above inductively we eventually obtain that $f m_{1}\in\Xi$.

Now for $\ell \geq 2$ we suppose that multiplication by any element of $A[t]$ on $m_i$ belongs to $\Xi$ for $1 \leq i \leq \ell-1$. We prove that $fm_{\ell} \in \Xi$ by the induction on the degree of $f$ in $t$, and note that the result is valid when $\deg_{t}f \leq s_{\ell}+\cdots+s_{r}-1$ by expanding $f$ in terms of powers of $(t-\theta)$. So we suppose that $\deg_{t}f\geq s_{\ell}+\cdots+s_{r}$.

We divide $f$ by $(t-\theta)^{s_{\ell}+\cdots+s_{r}}$ and write  \[f=g_{1}(t-\theta)^{s_{\ell}+\cdots+s_{r}}+\gamma_{1},\] where $g_{1},\gamma_{1}\in A[t]$ with $\deg_{t}\gamma_{1} < s_{\ell}+\cdots+s_{r}$. It follows that
\begin{align*}
fm_{\ell}&=g_{1}(t-\theta)^{s_{\ell}+\cdots+s_{r}}m_{\ell}+\gamma_{1} m_{\ell}\\
  &=g_{1}\left\{ \sigma m_{\ell} -H_{s_{\ell-1}-1}^{(-1)}(t-\theta)^{s_{\ell-1}+\cdots+s_{r}}m_{\ell-1} \right\}+\gamma_{1}m_{\ell} \\
  &=\begin{aligned}[t]
  g_{1}\Bigl\{ \sigma m_{\ell}& {}- \Bigl\{\sigma H_{s_{\ell-1}-1} m_{\ell-1} \\
  &{}- H_{s_{\ell-1}-1}^{(-1)}  H_{s_{\ell-2}-1}^{(-1)}(t-\theta)^{s_{\ell-2}+\cdots+s_{r}} m_{\ell-2} \Bigr\}  \Bigr\}+\gamma_{1}m_{\ell}
  \end{aligned} \\
  &\;\;\vdots\\
  &= g_{1}\left\{\sigma m_{\ell}+\sum_{i=1}^{\ell-1} (-1)^{i}\sigma \beta_{1}\cdots\beta_{i}m_{\ell-i}   \right\}+\gamma_{1}m_{\ell},
\end{align*}
where $\beta_{i}:=H_{s_{\ell-i}-1}\in A[t]$ for $i=1,\ldots,\ell-1$. It follows that
\[
fm_{\ell}=\sigma g^{(1)} m_{\ell}+ \sum_{i=1}^{\ell-1} (-1)^{i}\sigma g^{(1)} \beta_{1}\cdots\beta_{i}m_{\ell-i} +\gamma_{1}m_{\ell}.
\]
However, by expanding $\gamma_{1}$ in terms of powers of $(t-\theta)$ we see that $\gamma_{1}m_{\ell}\in \Xi$, and by hypothesis $\sum_{i=1}^{\ell-1} (-1)^{i}\sigma g^{(1)} \beta_{1}\cdots\beta_{i}m_{\ell-i}\in \Xi$. Thus, to prove the desired result we are reduced to proving that $g^{(1)} m_{\ell}\in A[t]$, which is valid by the induction hypothesis since $\deg_{t}g^{(1)}=\deg_{t}g < \deg_{t}f$.
\end{proof}

\begin{remark}\label{Rem:DeltaXi}
By \eqref{E:Delta} we see that $\Delta(\Xi)\subseteq E'(A)$.
\end{remark}

Now we put $\fQ=(H_{s_{1}-1},\ldots,H_{s_{r}-1})$, where $H_i$ are the Anderson-Thakur polynomials (see~\S\ref{sec:AT polynomials}).  We let $\bv_{\fs}\in E'(\ok)$ be image of $M$ under the composition of isomorphisms
\[
    \Ext_{\cF}^{1}({\mathbf{1}},M')   \cong  M'/(\sigma-1)M' \cong   E'(\ok) .
    \]
Precisely,
\[ \bv_{\fs}:=\Delta\left(H_{s_{r}-1}^{(-1)}(t-\theta)^{s_{r}}m_{r} \right) .\]
\begin{theorem}\label{T:Defined over A}
For each $r$-tuple $\fs=(s_{1},\ldots,s_{r})\in \NN^{r}$, we have that
\begin{enumerate}
\item[(a)] The associated $t$-module $E'$ given above is defined over $A$;
\item[(b)] The point $\bv_{\fs}$ is an integral point in $E'(A)$.
\end{enumerate}
\end{theorem}
\begin{proof} (a). Recall that $M'$ is the Frobenius module defined by $\Phi'$ as in  (\ref{E:Phi s'}) with $\ok[t]$-basis $m_{1},\ldots,m_{r}$. Put $d=(s_{1}+\cdots+s_{r})+\cdots+s_{r}$ and let $\left\{\nu_{1},\ldots,\nu_{d} \right\}$ be the $\ok[\sigma]$-basis of $M'$ given by \[\left\{(t-\theta)^{s_{1}+\cdots+s_{r}-1}m_{1},\cdots,(t-\theta)m_{1},m_{1},\ldots,(t-\theta)^{s_{r}-1}m_{r},\ldots,(t-\theta)m_{r},m_{r}\right\} .\] We identify $M'/(\sigma-1)M'$ with $\Mat_{d\times 1}(\ok)$ via the map $\Delta$ with respect to $\nu_{1},\ldots,\nu_{d}$.

Given any point $(a_{1},\ldots,a_{d})^{\tr}\in E'(\ok)$, its corresponding element in $M'/(\sigma-1)M'$ has a representative of the form $a_{1}\nu_{1}+\cdots+a_{d}\nu_{d}$. We claim that the element
\[ t\biggl(\sum_{i=1}^{d}a_{i}\nu_{i} \biggr) \]
 can be expressed as $\sum_{i=1}^{d}b_{i}\nu_{i}\in \Xi$ for which each $b_i$ is a of the form $b_{i}=\sum_{j}\sigma^{j}c_{j}$ so that $c_{j}$ is an $A$-linear combination of $q^{(\cdot)}$-th powers of the $a_{n}'s$. Then via the map $\Delta$, the claim implies that the $t$-module $E'$ is defined over $A$.

We observe that if some \[\nu_{i}\notin \mathcal{S}:=\left\{ (t-\theta)^{s_{1}+\cdots+s_{r}-1}m_{1},\ldots,(t-\theta)^{s_{r-1}+s_{r}-1}m_{r-1},(t-\theta)^{s_{r}-1}m_{r} \right\},\]
then \[ ta_{i}\nu_{i}=a_{i}(t-\theta)\nu_{i}+\theta a_{i}\nu_{i}=a_{i}\nu_{i-1}+\theta a_{i}\nu_{i}.\] Therefore we reduce the claim to the case of $\nu_{i}\in \mathcal{S}$.
To simplify the notation, we denote
\[ \nu_{i_1}=(t-\theta)^{s_{r}-1}m_{r},\quad\ldots, \quad\nu_{i_r}=(t-\theta)^{s_{1}+\cdots+s_{r}-1}m_{1}  .\]
Now given any $1\leq \ell \leq r$ we consider $t a_{i_{\ell}} \nu_{i_{\ell}}= a_{i_{\ell}} t (t-\theta)^{s_{1}+\cdots+s_{r}-1}m_{\ell}$. Applying Proposition~\ref{P:PropII} to $t (t-\theta)^{s_{1}+\cdots+s_{r}-1}m_{\ell}$ we see that $t a_{i_{\ell}} \nu_{i_{\ell}}$ can be written as the form
\[ a_{i_{\ell}}\sum_{j=1}^{d} \left(\sum_{e_{j}}\sigma^{e_{j}} b_{e_{j}}\right)\nu_{j}=\sum_{j=1}^{d}\left( \sum_{e_{j}} \sigma^{e_{j}} a_{i_{\ell}}^{q^{e_{j}}}b_{e_{j}}\right) \nu_{j}   \]
for some $b_{e_{j}}\in A$, whence the desired result follows.

(b).
Note that
\begin{align*}
 H_{s_{r}-1}^{(-1)}(t-\theta)^{s_{r}}m_{r}&=\sigma H_{s_{r}-1}m_{r}-H_{s_{r}-1}^{(-1)}H_{s_{r-1}-1}^{(-1)}(t-\theta)^{s_{r-1}+s_{r}} m_{r-1} \\
 &\;\;\vdots\\
 &= H_{s_{r}-1}^{(-1)}\left\{\sigma m_{r}+\sum_{i=1}^{r-1} (-1)^{i}\sigma \beta_{1}\cdots\beta_{i}m_{r-i}   \right\}\\
 &=\sigma H_{s_{r}-1} m_{r}+ \sum_{i=1}^{r-1} (-1)^{i}\sigma H_{s_{r}-1} \beta_{1}\cdots\beta_{i}m_{r-i},
\end{align*}
where $\beta_{i}:=H_{s_{r-i}-1}\in A[t]$ for $i=1,\ldots,r-1$. Applying Proposition~\ref{P:PropII} to the right-hand side of the equation above we see that \[H_{s_{r}-1}^{(-1)}(t-\theta)^{s_{r}}m_{r}\in \Xi.\] Since $\bv_{\fs}=\Delta\left(H_{s_{r}-1}^{(-1)}(t-\theta)^{s_{r}}m_{r}\right)$, the result follows from Remark~\ref{Rem:DeltaXi}.
\end{proof}

It follows that combining Theorems~\ref{T:MainThm},~\ref{T:Zeta-like},~\ref{T:Ext1}, and~\ref{T:t-module} we have the following criteria.
\begin{theorem}\label{T:MainThm2} For any $\fs=(s_{1},\ldots,s_{r})\in \NN^{r}$, we have the following equivalence.
\begin{enumerate}
\item
$\zeta_{A}(\fs)$ is Eulerian.
\item $\bv_{\fs}$ is an $\FF_{q}[t]$-torsion point in the $t$-module $E'(A)$.
\end{enumerate}
\end{theorem}

Finally Theorem~\ref{T:Zeta-like} can now be transformed into the following  concrete form:
\begin{theorem}\label{T:CriterionZeta-like}
Given $\fs=(s_{1},\ldots,s_{r})\in \NN^{r}$ with $w:=\sum_{i=1}^{r}s_{i}$, put $\fQ:=(H_{s_{1}-1},\ldots,H_{s_{r}-1})$ and $Q:=H_{w-1}$, where $\left\{ H_{n}\right\}$ be the Anderson-Thakur polynomials given in \S\ref{sec:AT polynomials}. Let $M\in \cF$ (resp. $N\in \cF$) be defined by $\Phi$ given in \eqref{E:Phi s} (resp. by the matrix given in \eqref{E:PhiN}). Let $M'\in\cF$ be defined by $\Phi'$ given in \eqref{E:Phi s'} with a $\ok[t]$-basis $\left\{m_{1},\ldots,m_{r} \right\}$, and $(E',\rho)$ be the $t$-module associated to $M'$. Put $\bv_{\fs}=\Delta\left(H_{s_{r}-1}^{(-1)}(t-\theta)^{s_{r}}m_{r}\right)\in E'(A)$ and $\bu_{\fs}=\Delta\left( H_{w-1}^{(-1)}(t-\theta)^{w}m_{1} \right)\in E'(A)$. Then we have:
\begin{enumerate}
\item[(a)] If $w$ is not divisible by $q-1$, then we have that $\zeta_{A}(\fs)$ is zeta-like if and only if there exists $a,b\in \FF_{q}[t]$ (not both zero) so that $\rho_{a}(\bv_{\fs})+\rho_{b}(\bu_{\fs})=0$ in the $t$-module $E'(A)$.
\item[(b)] If $w$ is divisible by $q-1$, then there exists nonzero $a\in \FF_{q}[t]$ so that $\zeta_{A}(\fs)$ is zeta-like if and only if $\rho_{a}(\bv_{\fs})=0$ in the $t$-module $E'(A)$.
\end{enumerate}
\end{theorem}

\begin{proof}
 Note that if $w$ is not divisible by $q-1$, then $\tilde{\pi}^{w}\notin k_{\infty}$ and hence the $k$-linear dependence of $\left\{ \zeta_{A}(\fs),\zeta_{A}(w),\tilde{\pi}^{w} \right\}$ is equivalent to that $\zeta_{A}(\fs)/\zeta_{A}(w)\in k$ . Thus, the result (a) follows from Theorem~\ref{T:Zeta-like} and the identification $\Ext_{\cF}^{1}({\bf{1}},M')\cong E'(\ok)$. When $w$ is divisible by $q-1$, we note that the zeta-like MZV's are the same as Eulerian MZV's because of (\ref{E:CarlitzFormula}), and hence the result (b) follows from Theorem~\ref{T:MainThm2}.
\end{proof}
\begin{remark}
In the case when the weight of $\zeta_{A}(\fs)$ is not divisible by $q-1$, the two integral points $\bu_{\fs}$ and $\bv_{\fs}$ are not $\FF_{q}[t]$-torsion elements inside $E'(A)$. If $\zeta_{A}(\fs)$ is zeta-like, $\FF_{q}[t]$-linear relations between $\bu_{\fs}$ and $\bv_{\fs}$ in Theorem~\ref{T:CriterionZeta-like}~(a) can be actually found. See \cite{KL15}.

\end{remark}

\section{The algorithm and rule specifying Eulerian MZV's} \label{sec:algorithm}

Fix the base finite field $\FF_q$, as we are in positive characteristic $p$, multizeta values satisfy $\zeta_{A}(\fs)^p=\zeta_{A}(s_{1},\dots,s_{r})^p=\zeta_{A}(ps_{1},\dots, ps_{r})$. Thus to investigate whether a given MZV is Eulerian we may restrict ourselves  to consider  only primitive tuples $\fs $, in the sense that not all $s_i$ are divisible by $p$. As first example of Eulerian MZV of depth $>1$, we cite e.g. Thakur~\cite[Thm.~5, Thm.~4]{Thakur09b}
\[\zeta_{A}(q - 1, (q-1)^2)={1\over{[1]^{q-1}} }\zeta_{ A}(q^2 - q),\]
where the Carlitz notation: $[\ell] := \theta^{q^\ell} - \theta$, is adopted, and  the depth two Eulerian MZV
\[\zeta_{A}(q^{\ell} - 1,  q^\ell(q-1))=\zeta_{A}(q^\ell - 1) \zeta_{A}(q - 1)^{q^\ell} - \zeta_{ A}(q^{\ell + 1} - 1).\]
This last relation has been extended inductively to arbitrary depth by Chen~\cite{Ch14} , yielding Eulerian MZV of arbitrary depth $r$ with respect to any $\FF_q$. See \eqref{I:idepth}.

Having Theorem~\ref{T:MainThm2} in our possession, we now write down an efficient algorithm for deciding whether any given MZV is Eulerian.

\subsection{The algorithm}
In accordance with Corollary~\ref{C:SimuEulerian}, we only consider the case of all $s_{i}$ divisible by $q-1$ when working on Eulerian MZVs. The following theorem offers an algorithm for Eulerian MZV's.
\begin{theorem}\label{T:EffectiveCriterion}
For any $\fs=(s_{1},\ldots,s_{r})\in \NN^{r}$ with all $s_{i}$ divisible by $q-1$,  we let \[w_{i}=s_{r-i}+s_{r-i+1}+\cdots+s_{r}\] for $i=1,\ldots,r-1$. Let $(E',\rho)$ be the $t$-module and $\bv_{\fs}$ be the integral point in $E'(A)$ given in Theorem~\ref{T:CriterionZeta-like}.  We decompose
\[ w_{i}=p^{\ell_{i}}n_{i}(q^{h_{i}}-1) \] so that $p\nmid n_{i}$ and $h_{i}$ is the greatest integer for which $q^{h_{i}}-1\mid w_{i}$. Put
\[ a=(t^{q^{h_{r-1}}}-t)^{p^{\ell_{r-1}}}\cdots (t^{q^{h_{1}}}-t)^{p^{\ell_{1}}}\frac{\Gamma_{s_{r}+1}}{\Gamma_{s_{r}}}{\rm{den}}(BC(s_{r}))|_{\theta=t} .\] Then we have that $\zeta_{A}(\fs)$ is Eulerian if and only if $\rho_{a}(\bv_{\fs})=0$.
\end{theorem}

\begin{proof}
Note that $(\Leftarrow)$ follows from Theorem~\ref{T:MainThm2}.
We prove the result $(\Rightarrow)$ by induction on the depth $r$. When $r=1$, we write $\fs=s\in \NN$. We claim that $\bv_{s}$ is essentially the same as the special point $Z_{s}$ in Definition~\ref{Def:special points}, and so the result is valid by Remark~\ref{Rem:torsion Zn}.
In this case, we note that $M'=C^{\otimes s}$ and $\bv_{s}=\Delta(H_{s-1}^{(-1)}(t-\theta)^{s})$. We further note that \[H_{s-1}^{(-1)}(t-\theta)^{s}\equiv H_{s-1}\hbox{ mod }(\sigma-1), \] which implies
\[\bv_{s}:=\Delta\left( H_{s-1}^{(-1)}(t-\theta)^{s}\right)=\Delta\left(H_{s-1}\right).\] Let $H_{s-1}=\sum_{i\geq 0}h_{si}\theta^{i}$ with $h_{si}\in \FF_{q}[t]$. Since the map $\Delta$ induces an $\FF_{q}[t]$-module isomorphism between $C^{\otimes s}/(\sigma-1)C^{\otimes s}$ and $\bC^{\otimes s}(\ok)$, and $\Delta$ maps $\theta$ to the vector $(0,\cdots,0,\theta^{i})^{\tr}\in \bC^{\otimes s}(A)$, we see that
\[\bv_{s}=\Delta\left( H_{s-1}\right)= \Delta\left(\sum_{i\geq 0}h_{si}\theta^{i} \right)= \sum_{i\geq 0}[h_{si}]_{n}\begin{pmatrix} 0\\ \vdots\\ 0 \\ \theta^{i}\end{pmatrix} =Z_{s} .\]
So the result is valid by Remark~\ref{Rem:torsion Zn}.

Suppose that the result is valid for depth less than $r$. Let $\Phi''$ be the square matrix of size $r-1$ cut from the right lower square of $\Phi'$ in (\ref{E:Phi s'}), and let $M''$ be the Frobenius module defined by $\Phi''$. Therefore we have the exact sequence of Frobenius modules
\[ 0\rightarrow C^{\otimes (s_{1}+\cdots+s_{r})}\rightarrow M'\twoheadrightarrow M''\rightarrow 0    .\]

For each $s\in\NN$, it is not hard to see that the $\FF_{q}[t]$-linear map $(\sigma-1):C^{\otimes s}\rightarrow C^{\otimes s}$ is injective, and hence arguments of induction on $r$ show that the $\FF_{q}[t]$-linear map $(\sigma-1):M''\rightarrow M''$ is also injective. It follows  that the snake lemma implies the exact sequence of $\FF_{q}[t]$-modules
\[  0\rightarrow C^{\otimes (s_{1}+\cdots+s_{r})}/(\sigma-1)C^{\otimes (s_{1}+\cdots+s_{r})} \rightarrow M'/(\sigma-1)M'\twoheadrightarrow M''/(\sigma-1)M''\rightarrow 0  \]
Denote by $(E'',\phi)$ the $t$-module underlying $M''/(\sigma-1)M''$, and so we have the exact sequence of $\FF_{q}[t]$-modules
\[   0\rightarrow \bC^{\otimes (s_{1}+\cdots+s_{r})}(\ok)\rightarrow E'(\ok)\twoheadrightarrow E''(\ok)\rightarrow 0  .\]
Denote by $\pi$ the projection map $E'(\ok)\twoheadrightarrow E''(\ok)$ given by
\[ (a_{1},\ldots,a_{d})^{\tr}\mapsto (a_{w+1},\ldots,a_{d})^{\tr} ,\]
where $d:=(s_{1}+\cdots+s_{r})+(s_{2}+\cdots+s_{r})+\cdots+s_{r}$ and $w:=\sum_{i=1}^{r}s_{i}$.

Put $\fs'=(s_{2},\ldots,s_{r})$. We claim that $\bv_{\fs'}=\pi(\bv_{\fs})$. Assume this claim first. We write $a=(t^{q^{h_{r-1}}}-t)^{p^{\ell_{r-1}}}b$, where
\[b:=(t^{q^{h_{r-2}}}-t)^{p^{\ell_{r-2}}}\cdots (t^{q^{h_{1}}}-t)^{p^{\ell_{1}}}\frac{\Gamma_{s_{r}+1}}{\Gamma_{s_{r}}}{\rm{den}}(BC(s_{r}))|_{\theta=t},\] then by the induction hypothesis we see that ${\bf{0}}=\phi_{b}(\bv_{\fs'})=\pi(\rho_{b}(\bv_{\fs}))$ and hence \[\rho_{b}(\bv_{\fs})\in\Ker \pi=\bC^{\otimes (s_1+\cdots+s_{r})}(\ok).\] Since by Theorem~\ref{T:Defined over A} $E'$ and $E''$ are defined over $A$ and $\bv_{\fs}, \bv_{\fs'}$ are integral points, $\rho_{b}(\bv_{\fs})\in \bC^{\otimes (s_{1}+\cdots+s_{r})}(k)_{{\rm{tor}}}$. Thus the result follows by Lemma~\ref{L:RationalTorsions}.

Finally, we note that the claim above follows from the following commutative diagram
\[
\xymatrix{
 M' \ar@{->>}[r]^{\Delta} \ar@{->>}[d] & E'(\ok) \ar@{->>}[d]^{\pi}\\
 M'' \ar@{->>}[r]^{\Delta} & E''(\ok)  , }
\]where $M'\twoheadrightarrow M''$ is the projection map given by $\sum_{i=1}^{r}f_{i}m_{i}\mapsto \sum_{i=2}^{r}f_{i}m_{i}$ with $f_{i}\in \ok[t]$.
\end{proof}

For any $\fs=(s_{1},\ldots,s_{r})\in \NN^{r}$, let $\bu=(u_{1},\ldots,u_{r})\in (k^{\times})^{r}\cap \mathbb{D}_{\fs}$ satisfy the hypotheses of Theorem~\ref{T:MainThmCMPL}. Applying the same arguments above we obtain the following result.

 \begin{corollary}
For any $\fs=(s_{1},\ldots,s_{r})\in \NN^{r}$ with all $s_{i}$ divisible by $q-1$, let $\bu=(u_{1},\ldots,u_{r})\in (k^{\times})^{r}\cap \mathbb{D}_{\fs}$ satisfy the hypotheses of Theorem~\ref{T:MainThmCMPL}. Put $w_{0}:=s_{r}$ and write \[ w_{0}=p^{\ell_{0}}n_{0}(q^{h_{0}}-1) \] so that $p\nmid n_{0}$ and $h_{0}$ is the greatest integer for which $q^{h_{0}}-1\mid w_{0}$. Let $(h_{1},\ell_{1}),\ldots,(h_{r-1},\ell_{r-1})$ be defined in Theorem~\ref{T:EffectiveCriterion}. Put $\fQ:=\bu$ and let $M'$ be the Frobenius module defined by the matrix \eqref{E:Phi s'} with a $\ok[t]$-basis $\left\{m_{1},\ldots,m_{r} \right\}$. Let $(E',\rho)$ be the $t$-module underlying associated to $M'$ and $\bv_{\fs}:=\Delta\left(u_{r}^{(-1)}(t-\theta)^{s_{r}}m_{r} \right) $. Define $a:=\prod_{i=0}^{r-1}(t^{q^{h_{i}}}-t)^{p^{\ell_{i}}}\in \FF_{q}[t]$. Then we have that the value $\Li_{\fs}(\bu)$ is Eulerian if and only if $\rho_{a}(\bv_{\fs})=0$.
 \end{corollary}

\begin{proof}
The proof is outlined as
\begin{enumerate}
\item[$\bullet$] The $t$-module $(E',\rho)$ is defined over $k$ and $\bv_{\fs}$ is rational point in $E'(k)$ using the fact  $\bu\in(k^{\times})^{r}\cap \mathbb{D}_{\fs}$ and following the arguments in Theorem~\ref{T:Defined over A}.
\item[$\bullet$] Note that for $r=1$, we have $E'=\bC^{\otimes s_{r}}$ from \S\ref{sec:Ext^1}.
\item[$\bullet$] In the case $r=1$, we have that $\bv_{\fs}$ is an $\FF_{q}[t]$-torsion point in $E'(k)$ if and only if $\bv_{\fs}$ is $(t^{q^{h_{0}}}-t)^{p^{\ell_{0}}}$-torsion by Lemma~\ref{L:RationalTorsions}.
\item[$\bullet$] The result follows by following the induction arguments in the proof of Theorem~\ref{T:EffectiveCriterion}.
\end{enumerate}
\end{proof}

\subsubsection{The algorithm}
 Here we provide the algorithm from Theorem~\ref{T:EffectiveCriterion}. Given any $\fs=(s_{1},\ldots,s_{r})\in \NN^{r}$ with each $s_{i}$ divisible by $q-1$, we list the essential steps as follows.
\begin{enumerate}
\item[(I)] Compute the Anderson-Thakur polynomials $H_{s_{1}-1},\ldots,H_{s_{r}-1}$.

\item[(II)] Put $\fQ=(H_{s_{1}-1},\ldots,H_{s_{r}-1})$ and let $M'$ be the Frobenius module defined by $\Phi'$ as in  (\ref{E:Phi s'}) with $\ok[t]$-basis $m_{1},\ldots,m_{r}$. Put $d=(s_{1}+\cdots+s_{r})+\cdots+s_{r}$ and let $\left\{\nu_{1},\ldots,\nu_{d} \right\}$ be the $\ok[\sigma]$-basis of $M'$ given by \[(t-\theta)^{s_{1}+\cdots+s_{r}-1}m_{1},\dots,(t-\theta)m_{1},m_{1},\ldots,(t-\theta)^{s_{r}-1}m_{r},\ldots,(t-\theta)m_{r},m_{r} .\] Identify $M'/(\sigma-1)M'$ with $\Mat_{d\times 1}(\ok)$ via $\nu_{1},\ldots,\nu_{d}$.

\item[(III)] Write down the $t$-action on $M'/(\sigma-1)M'$, and so giving a $t$-module structure on $\Mat_{d\times 1}(\ok)$, which we denote by $(E',\rho)$.

\item[(IV)] Consider $H_{s_{r}-1}^{(-1)}(t-\theta)^{s_{r}}m_{r}\in M'/(\sigma-1)M'$, which corresponds to an integral point $\bv_{\fs}=(a_{1},\ldots,a_{d})^{\tr}\in E'(A)$ from the decomposition $H_{s_{r}-1}^{(-1)}(t-\theta)^{s_{r}}m_{r}\equiv\sum_{i=1}^{d} a_{i} \nu_{i}$ (mod $\sigma-1$).

\item[(V)] Define the polynomial $a$ as in Theorem~\ref{T:EffectiveCriterion}, and then compute $\rho_{a}(\bv_{\fs})$. If it is zero, then $\zeta_{A}(s_{1},\ldots,s_{r})$ is Eulerian; otherwise, $\zeta_{A}(s_{1},\ldots,s_{r})$ is non-Eulerian.

\end{enumerate}

\subsubsection{Examples of $(E',\rho)$ and $\bv_{\fs}$}
We provide some examples of the explicit forms of $(E',\bv_{\fs})$. The following are two examples associated to Eulerian MZV's, i.e., $\bv_{\fs}$ is $\FF_{q}[t]$-torsion in $E'(A)$.

\begin{enumerate}
\item[(1)] Let $q=3,$ $\fs=(2,4)$. Then $(E',\rho)$ associated to $\zeta(2,4)$ is given by
\[
\rho_t=\left(\begin{smallmatrix}
\theta&1&0&0&0&0&0&0&0&0\\
0&\theta&1&0&0&0&0&0&0&0\\
0&0&\theta&1&0&0&0&0&0&0\\
0&0&0&\theta&1&0&0&0&0&0\\
0&0&0&0&\theta&1&0&0&0&0\\
\tau&0&0&0&0&\theta&2\tau&0&0&0\\
0&0&0&0&0&0&\theta&1&0&0\\
0&0&0&0&0&0&0&\theta&1&0\\
0&0&0&0&0&0&0&0&\theta&1\\
0&0&0&0&0&0&\tau&0&0&\theta\\
\end{smallmatrix}\right)
\]and
\[ \bv_{\fs}=(0,0,1,0,1,(\theta+2\theta^{3}),2,0,2,(2\theta+\theta^{3}))^{\tr}
  .\]

\item[(2)] Let $q=2$, $\fs=(1,2,4)$. Then $(E',\rho)$ associated to $\zeta(1,2,4)$ is given by
\[
\rho_t=\left(\begin{smallmatrix}
\theta&1&0&0&0&0&0&0&0&0&0&0&0&0&0&0&0\\
0&\theta&1&0&0&0&0&0&0&0&0&0&0&0&0&0&0\\
0&0&\theta&1&0&0&0&0&0&0&0&0&0&0&0&0&0\\
0&0&0&\theta&1&0&0&0&0&0&0&0&0&0&0&0&0\\
0&0&0&0&\theta&1&0&0&0&0&0&0&0&0&0&0&0\\
0&0&0&0&0&\theta&1&0&0&0&0&0&0&0&0&0&0\\
\tau&0&0&0&0&0&\theta&\tau&0&0&0&0&0&\tau&0&0&0\\
0&0&0&0&0&0&0&\theta&1&0&0&0&0&0&0&0&0\\
0&0&0&0&0&0&0&0&\theta&1&0&0&0&0&0&0&0\\
0&0&0&0&0&0&0&0&0&\theta&1&0&0&0&0&0&0\\
0&0&0&0&0&0&0&0&0&0&\theta&1&0&0&0&0&0\\
0&0&0&0&0&0&0&0&0&0&0&\theta&1&0&0&0&0\\
0&0&0&0&0&0&0&\tau&0&0&0&0&\theta&\tau&0&0&0\\
0&0&0&0&0&0&0&0&0&0&0&0&0&\theta&1&0&0\\
0&0&0&0&0&0&0&0&0&0&0&0&0&0&\theta&1&0\\
0&0&0&0&0&0&0&0&0&0&0&0&0&0&0&\theta&1\\
0&0&0&0&0&0&0&0&0&0&0&0&0&\tau&0&0&\theta\\
\end{smallmatrix}\right)
\] and  \[
{\bf{v}}_{\fs}=(0,0,0,0,1,1,(\theta+\theta^{2}),0,0,0,1,1,(\theta+\theta^{2}),0,1,1,(\theta+\theta^{2}))^{\tr}.
\]
\end{enumerate}

The following are two examples associated to non-Eulerian MZV's,  $\bv_{\fs}$ is not $\FF_{q}[t]$-torsion in $E'(A)$.
\begin{enumerate}
\item[(3)] Let $q=3$, $\fs=(4,2)$. Then $(E',\rho)$ associated to $\zeta(4,2)$ is given by
\[
\rho_t=\left(\begin{smallmatrix}
\theta&1&0&0&0&0&0&0\\
0&\theta&1&0&0&0&0&0\\
0&0&\theta&1&0&0&\tau&0\\
0&0&0&\theta&1&0&0&0\\
0&0&0&0&\theta&1&\tau&0\\
\tau&0&0&0&0&\theta&(\theta+2\theta^{3})\tau&0\\
0&0&0&0&0&0&\theta&1\\
0&0&0&0&0&0&\tau&\theta\\
\end{smallmatrix}\right)
\]and
\[ \bv_{\fs}=(0,0,1,0,1,(\theta+2\theta^{3\
}),0,1)^{\tr}
 .\]

\item[(4)] Let $q=3$, $\fs=(2,2,2)$. Then $(E',\rho)$ associated to $\zeta(2,2,2)$ is given by
\[ \rho_t=\left(\begin{smallmatrix}
\theta&1&0&0&0&0&0&0&0&0&0&0\\
0&\theta&1&0&0&0&0&0&0&0&0&0\\
0&0&\theta&1&0&0&0&0&0&0&0&0\\
0&0&0&\theta&1&0&0&0&0&0&0&0\\
0&0&0&0&\theta&1&0&0&0&0&0&0\\
\tau&0&0&0&0&\theta&2\tau&0&0&0&\tau&0\\
0&0&0&0&0&0&\theta&1&0&0&0&0\\
0&0&0&0&0&0&0&\theta&1&0&0&0\\
0&0&0&0&0&0&0&0&\theta&1&0&0\\
0&0&0&0&0&0&\tau&0&0&\theta&2\tau&0\\
0&0&0&0&0&0&0&0&0&0&\theta&1\\
0&0&0&0&0&0&0&0&0&0&\tau&\theta\\
\end{smallmatrix}\right) \]
and \[\bv_{\fs}=(0,0,0,0,0,1,0,0,0,2,0,1)^{\tr}. \]

\end{enumerate}

\subsection{Searching for the rules governing Eulerian MZV's}\label{sub:search}

Lara Rodr\'{i}guez and Thakur~\cite{LRT13} have given conjectures on which $r$-tuples $(s_1,\ldots,s_r)$ may occur for Eulerian MZV's. Furthermore, they also provided conjectural formulas for these special values.  Computations based on implementing the above algorithm in Magma (by Yi-Hsuan Lin) have led us to the following description of Eulerian MZV's in arbitrary depth.

Fix prime power $q$, and call the  sequence of $r$-tuples below Eulerian $r$-tuples with respect to $\FF_q$:
\[\Eu_1 := (q-1)\,\,\, \text{and}\,\,\, \Eu_{r+1} := (q - 1, q \Eu_r)  \in \NN^{r+1}.\]
For each depth $r$, we introduce a sequence of $r$-tuples in $\NN^r$ as follows:
\[\Eu_r(\ell) := (q^\ell - 1, q^\ell \Eu_{r-1}), \,\,\text{for }\,\,r> 1,\,\,\ell\ge 1, \]
and $\Eu_1(\ell) := (q^\ell -1)$.
Call this the canonical sequence of
depth $r$ with respect to $\FF_q$. The corresponding MZV's $\zeta_{A}(\Eu_r(\ell))$ are all Eulerian. This follows from the Euler-Carlitz formula\eqref{E:CarlitzFormula} and the following inductive formula of Chen~\cite{Ch14} for all $r\ge 2$ and $\ell\ge 1$:
\begin{equation}\label{I:idepth} \zeta_{A}(\Eu_r(\ell)) = \zeta_{A}(q^\ell -1) \zeta_{A}(\Eu_{r-1})^{q^\ell} - \zeta_{A}(\Eu_{r-1}(\ell + 1)).\end{equation}

Note that when $q=2$, all depth one Carlitz zeta values $\zeta_{A}(n)$ are Eulerian, and $\zeta_{A}(\fs)$ is Eulerian if and only if it is zeta-like. The following Eulerian multizeta values of arbitrary depth $r>1$ and weight $2^{r-1}$ have been found by
Lara  Rodr\'{i}guez and Thakur~\cite{LRT13}:
\begin{equation}\label{e1:excep}\zeta_{A}(1, \fs) := \zeta_{A}(1, 1, 2, \ldots, 2^{r-1}) = {1\over{[1]^{2^{r-1}} [2] ^{2^{r-2}}\cdots [r]}}\zeta_{A} (2^{r-1}). \end{equation}

For $q\ge 3$, we predict that the primitive Eulerian MZV's of depth $r\ge 2$ are precisely:
\begin{enumerate}
\item  The canonical family with every depth  $r\ge 2$, $\ell\ge 1$,
\[\zeta_{A}(\Eu_r(\ell)), \,\,\text{of weight}\,\,\, q^{r+\ell-1}-1,\]
\item An extra family in depth $r= 2$, $\ell\ge 1$,
\[\zeta_{A}(q^\ell (q - 1), q^{\ell + 2} - 1 - q^\ell (q - 1)),\,\,\text{of weight}\,\,\, q^{\ell+2}-1. \]
\item An exceptional primitive Eulerian MZV  in depth $r= 2, $
\[\zeta_{ A}(q - 1, (q-1)^2),\,\,\text{of weight}\,\,\, q^2-q. \]
\end{enumerate}
Thus for $q\ge 3$ there should exist depth $r >1$ primitive Eulerian MZV's only in weights $q^2-q$ (depth 2), and $q^\ell - 1$ (in any depth) for $\ell\ge r.$ For depth $r=2$, only in weights $q^2-q$, $q^2 -1$, and each weight has only one primitive Eulerian MZV. For weight $q^\ell -1$, $\ell\ge 3$, each weight has two primitive Eulerian MZV's, coming from the two families in (1) and (2). Lara Rodr\'{i}guez and Thakur~\cite{LRT13} have also given precise formulas (valid for any $q$) for the family (2):
\[\zeta_{A}(q^\ell (q - 1), q^{\ell + 2} - 1 - q^\ell (q - 1))=  {1\over{[1]^{q^\ell (q-1)}}}\zeta_{A}(q^{\ell + 2} - 1). \]

In the case $q=2$, we predict that the primitive Eulerian MZV's are given by:
\begin{enumerate}
\item The canonical family with every depth  $r\ge 2$, $\ell\ge 1$,
\[\zeta_{A}(\Eu_r(\ell)), \,\,\text{of weight}\,\,\, 2^{r+\ell-1}-1.\]
\item The extra family in depth $r= 2$, $\ell\ge 1$,
\[\zeta_{A}(2^\ell , 2^{\ell + 2} - 1 - 2^\ell)\,\,\text{weight}\,\,\, 2^{\ell+2}-1.\]
\item Three exceptional primitive Eulerian MZV  in depth $r= 2$:
\[\zeta_{A}(1, 1), \,\,\, \zeta_{A}(1, 3)=\zeta_{A}(1, 2^r - 1) = \biggl(\frac{1}{[1] [2]}+ \frac{1}{[1]} \biggr) \zeta_{A}(4), \,\,\text{and}\]
\[\zeta_{A}(3, 5) =  \frac{[2]^2 +1}{[1]^4[2]}\zeta_{A}(8).\]
Thus primitive Eulerian pairs exist
only in weights, $2$, $3$, $4$, $7$, $8$, $2^\ell - 1$, $\ell\ge 4$. For each weight $2^\ell - 1$, $\ell\ge 3$, there are
exactly two primitive Eulerian pairs from the two families. For weights $2$, $3$, $4$, $8$, each weight has only
one primitive Eulerian pair.
\item There are  exceptional primitive Eulerian MZV's  for depth $r>2$ :
$\zeta_{A}(1, \fs)$, where $\fs$ is a primitive Eulerian tuple of depth $r-1$ and weight either $2^r $ or $2^{r-1}$. Thus for depth $r> 3$, in each weight $2^r$, $2^{r- 1}$ there is only one exceptional
primitive Eulerian MZV. In the case of depth $3$, there are two exceptional primitive Eulerian of weight $8$, and one exceptional primitive Eulerian of weight $4$.

The exceptional sequence of primitive Eulerian MZV's of weight $2^{r-1}$, $r>1$, is the one given in \eqref{e1:excep}. That the above exceptional sequence of primitive MZV's of weight $2^r$ consists only of  Eulerians (first conjectured by  Lara Rodr\'{i}guez and Thakur~\cite{LRT13}) is a consequence of  the following formula of Chen~\cite{Ch14} :
\[\zeta_{A}(1, \fs) = \zeta_{A}(1, 3, 2^2, \ldots, 2^{r-1}) = \zeta_{A}(1) \zeta_{A}(1, 2, \ldots, 2^{r-1}) +  \zeta_{A}(1, 1, 2, \ldots, 2^{r-2}) ^2.\]
When $r=2$, this last equality also goes back to Thakur~\cite[Thm.~ 8]{Thakur09b}
\end{enumerate}

All MZV's in the above list have been confirmed to be Eulerian by ~\cite{Ch14} and ~\cite{LRT13}.
Our computations  suggest that the above list exhausts all primitive Eulerian multizeta values for $\FF_q[\theta]$. In other words, any $r$-tuple $\fs$ of depth $r>1$, not accounted by our list above should give non-Eulerian $\zeta_{A}(\fs)$. Previously in \cite{LRT13},  Lara Rodr\'{i}guez and Thakur had also collected data  basing on continued fraction computations  to decipher  the occurrence of Eulerian MZV's, and made precise conjectures characterizing Eulerian tuples. Their conjectures agree with the above list. Our ``$t$-motivic'' algorithm for determining Eulerian multizeta values is rooted by an entirely different principle, runs a bit more efficiently and is completely algebraic. It allows us to do computations inductively for higher depth because of the key Corollary~\ref{C:SimuEulerian}, thereby leading to the above description which we believe is a complete list.

\begin{center}
\textbf{Summary of data certified by our computations}
\end{center}
All tuples $\fs$ of depth $r$ and weights $w$ within the following respective bounds have been checked for the Eulerian property.  The answers agree with the description above and the MZV's in the complementary part of the list above are non-Eulerian. When $3\leq q\leq 11$, all tuples having their weights within the bounds below have been checked, with no restriction on their depths except $q=2$.
$$q=2, \,\,\, \text{depth}=2,\,\,\,  \text{weight} \le 256$$
$$q=2, \,\,\, \text{depth}=3,4,5,\,\,\,  \text{weight} \le 128   $$
$$q=2, \,\,\, \text{depth}=6,\,\,\,  \text{weight} \le 64   $$
$$q=3, \,\,\,  \,\,\, \text{weight} \le 243$$
$$q=4, \,\,\,   \,\,\, \text{weight} \le 256$$
$$q=5, 7 \,\,\,  \,\,\,  \text{weight} \le q^3$$
$$8\le q\le 23, \,\,\,  \,\,\,  \text{weight} \le q^2$$

\bibliographystyle{alpha}

\end{document}